\documentclass[12pt]{article}

\usepackage{graphicx,hyperref}
\usepackage{amsmath,amsthm,amssymb,color}
\usepackage[noadjust,sort]{cite}
\usepackage{tocloft}
\usepackage{subdepth}

\makeatletter
\g@addto@macro\normalsize{%
  \setlength\abovedisplayskip{8pt plus 3pt minus 3pt}
  \setlength\belowdisplayskip{8pt plus 3pt minus 3pt}
  \setlength\abovedisplayshortskip{6pt plus 3pt minus 2pt}
  \setlength\belowdisplayshortskip{6pt plus 3pt minus 2pt}
}
\makeatother

\date{}

\normalsize

\numberwithin{equation}{section}

\def\({\bigl(}
\def\){\bigr)}

\newtheorem{thm}{Theorem}[section]
\newtheorem{cor}[thm]{Corollary}
\newtheorem{lemma}[thm]{Lemma}

\theoremstyle{definition}
\newtheorem{remark}[thm]{Remark}

\def\abs#1{\lvert#1\rvert} \let\card=\abs
\def\Abs#1{\bigl\lvert#1\bigr\rvert} 
\def\norm#1{\lVert#1\rVert}
\def\maxnorm#1{\norm{#1}_{\mathrm{max}}}
\def\interior{\operatorname{int}}
\def\dmax{d_{\mathrm{max}}}
\def\dmin{d_{\mathrm{min}}}
\def\smax{s_{\mathrm{max}}}
\def\dfrac#1#2{\lower0.15ex\hbox{\large$\textstyle\frac{#1}{#2}$}}
\def\({\bigl(}
\def\){\bigr)}
\def\st{\mathrel{|}}
\def\St{\bigm|}

\newcommand{\stirlingii}{\genfrac{\{}{\}}{0pt}{}}

\let\eps=\varepsilon

\def\P{\boldsymbol{P}}

\def\S{\boldsymbol{S}}

\def\X{\boldsymbol{X}}
\def\Y{\boldsymbol{Y}}
\def\Z{\boldsymbol{Z}}
\def\x{\boldsymbol{x}}
\def\calF{\mathcal{F}}
\def\calG{\mathcal{G}}
\def\calH{\mathcal{H}}
\def\F{\boldsymbol{\calF}}

\def\calB{\mathcal{B}}
\def\thetavec{\boldsymbol{\theta}}
\def\dvec{\boldsymbol{d}}
\def\wvec{\boldsymbol{w}}
\def\xvec{\boldsymbol{x}}
\def\yvec{\boldsymbol{y}}
\def\zvec{\boldsymbol{z}}
\def\alphavec{\boldsymbol{\alpha}}
\def\nperp{n_{\scriptscriptstyle\perp}}

\def\xperp{\boldsymbol{x}_{\scriptscriptstyle\perp}}
\def\xpar{\boldsymbol{x}_{\scriptscriptstyle\parallel}}
\def\trans{^{\mathrm{T}}}

\def\E{\operatorname{\mathbb{E}}}
\def\V{\operatorname{\mathbb{V\!}}}

\def\EFj#1#2{\E_{#2}(#1)}
\def\EFjb#1#2{\E_{#2}#1}
\def\Var{\operatorname{Var}}
\def\Cov{\operatorname{Cov}}

\def\VFjb#1#2{\V_{#2}#1}
\def\Prob{\operatorname{Prob}}
\def\Diam{\operatorname{diam}}

\def\Reals{{\mathbb{R}}}
\def\Complexes{{\mathbb{C}}}

\def\Naturals{{\mathbb{N}}}

\def\esssup{\operatorname{ess\,sup}\limits}
\def\Hempty{(\emptyset,\emptyset)}

\def\nicebreak{\vskip 0pt plus 50pt\penalty-300\vskip 0pt plus -50pt }

\begin{document}

\title{Complex martingales and asymptotic enumeration}

\author{Mikhail Isaev${}^*{}^\dag$~~and~~Brendan~D.~McKay\vrule width0pt height2ex\thanks
 {Research supported by the Australian Research Council.}\\
\small ${}^*$Research School of Computer Science\\[-0.9ex]
\small Australian National University\\[-0.9ex]
\small Canberra ACT 2601, Australia\\[0.3ex]
\small ${}^\dag$Moscow Institute of Physics and Technology\\[-0.9ex]
\small Dolgoprudny, 141700, Russia\\[-0.3ex]
\small\texttt{isaev.m.i@gmail.com, brendan.mckay@anu.edu.au}
}

\maketitle

\begin{abstract}
Many enumeration problems in combinatorics, including such
fundamental questions as the number of regular graphs,
can be expressed as high-dimensional complex integrals.
Motivated by the need for a systematic study of the asymptotic
behaviour of such integrals, we establish explicit bounds on the
exponentials of complex martingales.  Those bounds applied
to the case of truncated normal distributions are precise enough
to include and extend many enumerative results 
of Barvinok, Canfield, Gao, Greenhill, Hartigan,
Isaev, McKay, Wang, Wormald, and others. 
Our method applies to sums as well as integrals.

As a first illustration of the power of our theory, we 
considerably strengthen existing results on the relationship
between random graphs or bipartite graphs with specified
degrees and the so-called $\beta$-model of random graphs
with independent edges, which is equivalent to the
Rasch model in the bipartite case.
\end{abstract}

\nicebreak
\section{Introduction}\label{S:intro}

A large number of combinatorial enumeration problems can
be expressed in terms of high-dimensional integrals, often,
but not always, resulting from Fourier inversion applied
to a multivariable generating function.

To illustrate what we mean, here are two examples. The number
of undirected simple graphs with degrees $d_1,\ldots,d_n$ is
given by
\begin{equation}\label{allgraphs}
  \frac{1}{(2\pi i)^n} \oint\!\cdots\!\oint\;
       \frac{\prod_{1\le j<k\le n} (1 + z_jz_k)}
   {z_1^{d_1+1}\cdots z_n^{d_n+1}}
        \, dz_1\cdots dz_n,    
\end{equation}
while the number of $m\times n$ nonnegative integer matrices
(contingency tables) with row sums $r_1,\ldots,r_m$ and
column sums $c_1,\ldots,c_n$ is given by
\begin{equation}\label{ct}
  \frac{1}{(2\pi i)^{m+n}} \oint\!\cdots\!\oint\;
       \frac{\prod_{1\le j\le m,1\le k\le n} (1 - w_jz_k)^{-1}}
   {w_1^{r_1+1}\cdots w_m^{r_n+1}\,z_1^{c_1+1}\cdots z_n^{c_n+1}}
        \, dw_1\cdots dw_m\,dz_1\cdots dz_n,    
\end{equation}
where each contour encloses the origin once anticlockwise.
Although explicit evaluation of such integrals is rarely
possible, under some circumstances asymptotic estimation
is tractable.
This was first achieved by McKay and Wormald in 1990,
for~\eqref{allgraphs} in the case of degree sequences not
far from regular~\cite{MWreg} and some classes of digraphs
that include regular tournaments~\cite{Mtourn}.

Since then, many other examples have appeared that include
classes of 0-1 matrices
\cite{Barv01,BarvHart1,CM,CGM,GMX,ranx,Ord};
directed graphs by degree sequence
\cite{GMW,GMX,Mtourn,MWtournament,Wang3,Wang2};
eulerian digraphs \cite{MIorient,Wang1};
eulerian circuits \cite{euler,MIeuler,Isaeva,MImix}; types of
integer matrices \cite{CMint,MM,BarvHart2}; and
multiple other problems~\cite{CI,deLauney,Kuperberg,Mont}. 
The method often gives a surprisingly good approximation even
for structures of moderate size~\cite{CM,CMint,loopy,Isaeva2,MM,euler}.

\medskip
Estimation of integrals like~\eqref{allgraphs} and~\eqref{ct}
involves several steps, none of them trivial.
\begin{itemize}\itemsep=0pt
  \item[(a)] Choose as contours circles $r_je^{i\theta_j}$ whose
      radii are chosen so that they pass together through the
      saddle-point (or close enough to it).
      This involves solving nonlinear equations
      or maximizing an entropy function.
  \item[(b)] Identify one or more small regions (in $\{\theta_j\}$-space)
      in which the value of the integral is concentrated.
      This might be small boxes enclosing two points (as in~\eqref{allgraphs})
      or the neighbourhood of a low-dimensional subspace (as in~\eqref{ct}).
  \item[(c)] Within those small regions, approximate the 
     integrand by a more tractable function and estimate
     its integral.
\end{itemize}
The present paper is motivated by step~(c).
The integrals that occur are typically of the form
\[
   I = \int_B \exp\(-\xvec\trans\!A\,\xvec + f(\xvec)\)
      \, d\xvec,
\]
where $B$ is a region containing the origin, 
$A$ is a positive-semidefinite real matrix, and $ f(\xvec)$
is a function well-approximated by a truncated Taylor
series with complex coefficients. 
The matrix $A$ might not be of full rank.

Now let $\X$ be a random variable whose distribution
is given by the gaussian density $C\exp(-\xvec\trans\!A\xvec)$
truncated to domain~$B$, where $C$ is the normalising
constant.
Then, by the definition of expectation,  we have
\[
    I  = \frac{ \int_B \exp\(-\xvec\trans\!A\,\xvec + f(\xvec)\)
      \, d\xvec}{C\int_B \exp\(-\xvec\trans\!A\,\xvec\)
      \, d\xvec} =      
      C^{-1}\E e^{f(\X)},
\]
so the problem is reduced to estimating $\E e^{f(\X)}$.
Our main aim is to make estimation of such integrals more
systematic by providing some general theory about $\E e^{f(\X)}$.

We will give explicit bounds on  $\E e^{f(\X)}$ that are general and
precise enough to cover and generalize
the steps corresponding to (c) in all of the examples
listed above and many more similar examples.
In fact, we will not restrict ourselves to truncated gaussian
measures or to functions $f$ that are approximated by polynomials.
Furthermore, both our measure and our functions $f$ can be either
smooth or discrete, allowing for sums as well as integrals.

\nicebreak
\subsection{Summary of the paper}\label{S:mainthms}

Section~\ref{S:martin} gives our main theorem in its
most general form, providing explicit bounds on 
$\E e^{Z_n}$ when $Z_0,\ldots,Z_n$ is a complex
martingale, based on  properties of the martingale differences.
Section~\ref{S:fXY} applies the martingale theorems to
functions of independent random variables, via the Doob
martingale. We also show how to bound the necessary parameters
for smooth functions and how to handle vector measures whose
components are independent only when the measure is rotated.

Section~\ref{S:gauss} considers the case of gaussian measures
which are truncated to a finite region (usually a cuboid,
perhaps intersected with a linear subspace). These are the
theorems which can be applied directly to the enumeration
problems we have surveyed.  The cases of full-rank and
non-full-rank gaussians are somewhat different.  Finally
in that section we give some lemmas useful for managing
the quadratic forms which occur.

In Section~\ref{S:examples} we demonstrate the power of our
theorems using the example of graphs or bipartite graphs with
given degrees. In each case, we allow degree sequences as
general as those allowed by Barvinok and Hartigan~\cite{BarvHart1},
but we also allow a moderate number of forced and forbidden
edges.  This permits us to prove, in Section~\ref{S:concentration},
that the corresponding $\beta$-models are closer than
previously known to the uniform model of random graphs with
given degrees.

The Appendix collects some technical lemmas we need in the proofs.

\renewcommand{\contentsname}{}
\tableofcontents

\nicebreak
\section{The exponential of a complex martingale}\label{S:martin}

In this section we state and prove our theorems in their most general forms.

Let $\P=(\varOmega,\calF,P)$ be a probability space.
We are interested in estimates for the expected value of  
$e^Z$, where~$Z$ is
a complex-valued random variable on~$\P$.
Such estimates for the case of real~$Z$ are commonplace
as intermediate steps towards concentration inequalities,
such as in the classical works of Hoeffding and McDiarmid
\cite{Hoeffding,McDiarmid}.  However, we seek~$\E e^Z$ itself
and few such intermediate results carry over unchanged to
the complex case,
perhaps fundamentally due to the non-convexity of the
exponential function in the complex plane.

As our primary measure of spread of a complex random
variable we use the diameter of its essential support.
This choice was inspired by its effective use (in the real
case) by McDiarmid~\cite[Theorem 3.1]{McDiarmid}
in analysing the concentration of functions
of independent random variables.
A bound on $\abs{f(\x')-f(\x)}$ whenever
$\x,\x'$ differ only in the $k$-th position is,
roughly speaking, the same as
a bound on the diameter of the random variable
$f(x_1,\ldots,x_{k-1},X_k,\allowbreak x_{k+1},\ldots,x_n)$
for constant $x_1,\ldots,x_{k-1},x_{k+1},\ldots,x_n$.

Note that having diameter $\alpha$ is weaker than being
confined to a disk of diameter~$\alpha$.  So while we could
alternatively have generalized real intervals into complex disks,
doing so would weaken our theorems.

In the next subsection we define the diameter formally, including
a conditional version, and prove some properties that we will need.
Then, in two further subsections, we use the diameter to bound
the exponential of a complex martingale.

Recall that for complex random variables $Z$ there are two types of
squared variation commonly defined.  The variance is
\begin{align*}
\Var Z &= \E\,\abs{Z-\E Z}^2 = \E\,\abs{Z}^2 - \abs{\E Z}^2
=\Var\Re Z+\Var\Im Z, \\
\intertext{while the pseudovariance is}
\V Z &= \E\,(Z-\E Z)^2 = \E Z^2 - (\E Z)^2
= \Var\Re Z - \Var\Im Z + 2i\Cov(\Re Z,\Im Z).
\end{align*}
Of course, these are equal for real random variables.

\subsection{The diameter of a complex random variable}\label{S:diam}

Let
$X$ be an a.s.~bounded real random variable on~$\P = (\varOmega,\calF,P)$.
As usual, define the \textit{essential supremum} of~$X$ as
\[ \esssup X = \sup\, \bigl\{ x\in\Reals \St P(X>x) > 0 bigr\}. \]
If $\abs X\le c$ a.s., it is well-known that
$\esssup X = -c + \lim_{r\to\infty} \(\E ((X+c)^r)\)^{1/r}$.
If $Z$ is an a.s.~bounded complex random variable on~$\P$ then
we define the \textit{diameter} of~$Z$ to be
\begin{equation}\label{diam1}
   \Diam Z = \esssup\,\abs{Z-Z'}, 
   \quad\text{where $Z'$ is an independent copy of~$Z$}.
\end{equation}
The probability in~\eqref{diam1} is interpreted in the
product space $\P\otimes\P$ in the standard fashion.
We will also use an equivalent definition that does not
use the product space.
Given an angle~$\theta$, the extent of $Z$ in the $\theta$
direction (i.e., the inner product of $Z$ with the unit
vector in the $\theta$ direction),
is $\Re(e^{-i\theta}Z)$, so we can alternatively define
\begin{equation}\label{diam2}
   \Diam Z = \sup_{\theta\in (-\pi,\pi]}\, 
     \( \esssup(\Re(e^{-i\theta}Z))
         + \esssup(-\Re(e^{-i\theta}Z)) \).
\end{equation}

\begin{remark}
To see that~\eqref{diam1} and~\eqref{diam2} are equivalent,
suppose first that $\Diam Z > d+\eps$ according to~\eqref{diam1},
for some $\eps>0$.
Assuming that $\abs Z\le c$ a.s., cover the disk
$\{z\st \abs z \le c\}$ by finitely many open disks of radius $\eps/4$.
If for each pair $D,D'$ of such disks whose centres are at
least $d+\eps/2$ apart we have $P(Z\in D,Z'\in D')=0$, then
$\esssup\,\abs{Z-Z'}\le d+\eps$, a contradiction. So
choose two of the disks, $D,D'$, with centres at least
$d+\eps/2$ apart, such that $P(Z\in D,Z'\in D')=P(Z\in D)\,P(Z'\in D')>0$.
Taking $\theta$ to be the direction from the centre of
$D$ to the centre of $D'$, we find that $\Diam Z\ge d$
according to~\eqref{diam2}.  Conversely, if there is
$\theta$ such that the argument of the $\sup_\theta$ in~\eqref{diam2}
is greater than~$d$, there are half-planes more than~$d$ apart
in each of which $Z$ has nonzero probability, proving that
$\Diam Z>d$ according to~\eqref{diam1}.
\end{remark}

The basic properties of the diameter of a complex random variable are
given by the following lemma.

\begin{lemma}\label{Lemma_diam}
	 Let $Z$ be an a.s.\ bounded complex random variable on~$\P$. Then,
  \begin{itemize}\itemsep=0pt
   \item[(a)] $\Diam Z = 0$ iff $Z$ is a.s.\ constant. 
   \item[(b)] $\Diam\,(aZ+b) = \abs{a}\Diam Z$ for any $a,b \in \mathbb{C}$. 
    \item[(c)] $\Diam\, (Z + W) \leq \Diam Z + \Diam W$ 
   for any a.s.\ bounded complex random variable $W$ on~$\P$.
    \item[(d)] $\Diam \Re Z \leq \Diam Z \leq 2\esssup\, \abs{Z}$.
   \item[(e)] $\Diam \,(\Re Z)^2 \leq \Diam Z^2\le  2\esssup\,\abs{Z}\cdot\Diam Z$.
   
  \item[(f)] $\abs{Z-\E Z} \le \Diam Z$ a.s.
  \item[(g)] There exists $a\in \Complexes$ such that $\abs{Z-a} \le \tfrac{1}{\sqrt{3}}\Diam Z$ a.s.
  \item[(h)]  $\Var Z =\E\, \abs{Z-\E Z}^2 \leq \tfrac13 (\Diam Z)^2$ 
  and $\abs{\V Z}= \abs{\E\, (Z-\E Z)^2}  \leq \tfrac14 (\Diam Z)^2$.
  \end{itemize}
\end{lemma}

\begin{proof}
Claims (a),(b) follow immediately from Definition \eqref{diam1}. We get  claim (c) from  Definition   \eqref{diam2} and the fact 
that $\esssup (X+Y) \leq \esssup X + \esssup Y$ for any a.s. bounded real random variables $X,Y$ on $\P$.

  Let $Z'$ be an independent copy of~$Z$. We note then (almost surely) that
    \begin{align*}
  	\abs{\Re Z - \Re Z'} &\leq \abs{Z-Z'} \leq \esssup\, \abs{Z - Z'} = \Diam Z,  \\
  	\abs{(\Re Z)^2 - (\Re Z')^2} 
	   &= \abs{\Re (Z - Z')}\cdot\abs{\Re (Z + Z')}
	      \leq \abs{Z - Z'}\cdot\abs{Z + Z'}\\
	     &\hspace{3.4cm} \leq \esssup \,\abs{Z^2 - (Z')^2}  = \Diam Z^2.\\
	   \abs{Z - Z'} &\leq  \esssup\, \abs{Z} + \esssup\, \abs{Z'} 
	   = 2\esssup\, \abs{Z}.\\
	\Abs{ Z^2 - (Z')^2 }
  &\le \esssup\, \abs{Z+Z'}\cdot
     \esssup \,\abs{Z - Z'}  \leq 2 \esssup \,\abs{Z} \cdot \Diam Z.
\end{align*}
 Due to Definition \eqref{diam1}, claims (d) and (e) follow. 
 
 Using  Definition \eqref{diam2}, the fact  that 
 $\abs{X-EX} \leq \esssup(X) - \esssup(-X)$ a.s. for any a.s.~bounded real
  random variable $X$ on $\P$ and the equation
 \[
 	\abs{Z-\E Z} 
 	= \sup_{\theta\in (-\pi,\pi]} \Abs{\Re(e^{-i\theta} (Z-\E Z))},
 \]
  we obtain claim (f).

Claim (g) follows from a standard result on convex sets, see \cite[Thm.~12.3]{Lay} for example.
An equilateral triangle shows that the constant cannot be reduced. 
To prove the first part of claim (h), note that $\Var Z = \Var (Z-a) \leq \esssup\, \abs{Z-a}^2$. 

However, for any a.s. bounded real random variable $X$ on $\P$ 
\[
 \Abs{X - \tfrac12(\esssup X + \esssup(-X))} \leq \tfrac12( \esssup X - \esssup(-X)) \ \text{ a.s.},
\]
which implies $\Var X \leq \tfrac{1}{4} (\Diam X)^2$.  To prove the second part of claim (h), note that
\[
	\abs{\E (Z-\E Z)^2} \leq  \E \(\Re(e^{-i\theta} (Z-\E Z))^2 \) = \Var \Re(e^{-i\theta} Z ),
\]
where $e^{i\theta}  = \E (Z-\E Z)^2 / \abs{\E (Z-\E Z)^2}$, and 
$\Diam \( \Re(e^{-i\theta} Z)\)  \leq \Diam Z$ on account of claims~(b) and~(d).
\end{proof}

\smallskip
We will also use a conditional version of the diameter.
Let $\calG\subseteq\calF$ be a $\sigma$-field.
For a real random variable $X$ on $\P=(\varOmega,\calF,P)$
such that $\abs X\le c$ a.s.,
we can define the \textit{conditional essential supremum}
of $X$ to be the $\calG$-measurable function
\begin{equation}\label{def_esssup}
   \esssup\,(X\st\calG) 
     = -c + \lim_{r\to\infty}\, \(\E((X+c)^r\St\calG)\)^{1/r}.
\end{equation}
Alternative equivalent definitions and many properties of the
conditional essential supremum are given in~\cite{Barron}.
Informally,
$\esssup\,(X\st\calG)$ is the least $\calG$-measurable
function $G:\varOmega\to\Reals$ such that $X\le G$~a.s.
Now we can extend~\eqref{diam2} to define the
\textit{conditional diameter}:
\begin{equation}\label{diam3}
   \Diam (Z\st\calG) = \sup_{\theta\in (-\pi,\pi]}\, 
     \( \esssup(\Re(e^{-i\theta}Z)\st\calG)
         + \esssup(-\Re(e^{-i\theta}Z)\st\calG) \).
\end{equation}
Note that $\Diam (Z\st\calG)$ is a function from $\varOmega$
to $\Reals_+$.
For any $\omega\in\varOmega$, the argument of the $\sup_\theta$
in~\eqref{diam3} is a continuous function of $\theta$ (since
$Z$ is a.s.\ bounded), so the supremum over $\theta$ is the
same if restricted to a dense countable subset of~$(-\pi,\pi]$.
This proves that $\Diam (Z\st\calG)$ is $\calG$-measurable.

If $Z$ is real, we can restrict~\eqref{diam3} to $\theta=0$
and then $\Diam(Z\st\calG)$ is the same as the conditional
range defined by McDiarmid~\cite[Sec.\,3.4]{McDiarmid}.

Now let $P_{Z\st\calG}:\calB(\Complexes)\times\varOmega\to[0,1]$
be a regular conditional distribution
for $Z$ given $\calG$, where $\calB(\Complexes)$ is
the Borel field of $\Complexes$.
That is, for each $\omega\in\varOmega$, $P_{Z\st\calG}(\cdot,\omega)$
is a probability measure on $\calB(\Complexes)$, and for each $A\in\calB(\Complexes)$, $P_{Z\st\calG}(A,\cdot)$ is $\calG$-measurable
and $P_{Z\st\calG}(A,\cdot)=P(Z^{-1}(A)\st\calG)$~a.s.
For the existence of $P_{Z\st\calG}$ and basic theory, see~\cite[Chap.~6]{Kallenberg}.

For each $\omega\in\varOmega$, let $K_\omega(Z\st\calG)$
be the class of random variables from $\varOmega$ to $\Complexes$
that induce the distribution
$P_{Z\st\calG}(\cdot,\omega)$ on $\calB(\Complexes)$.
The most important property of $K_\omega(Z\st\calG)$ is:

\begin{lemma}\label{Komega}
  Let $\calG\subseteq\calF$ be a $\sigma$-field and 
  $Z$ be an a.s.\ bounded complex random variable on~$\P$.
  Let $Z_\omega$ be an arbitrary member of 
  $K_\omega(Z\st\calG)$ for each $\omega\in\varOmega$.
  Let $W$ be a $\calG$-measurable random variable on~$\P$,
    and let
    $\phi:\Complexes\times W(\varOmega)\to\Complexes$ be a measurable
    function such that $\E\,\abs{\phi(Z,W)}<\infty$.
    Then, for almost all $\omega\in\varOmega$,
    \begin{equation}\label{Ephi}
      \E(\phi(Z,W)\st\calG)(\omega) = \E\,\phi(Z_\omega,W),
    \end{equation}
    and moreover
    $\phi(Z_\omega,W)\in K_\omega(\phi(Z,W)\st\calG)$.
   Also the random variable $\omega\mapsto \E\phi(Z_\omega,W)$ is
   $\calG$-measurable. Consequently, for almost all $\omega\in\varOmega$,
   	\begin{equation}\label{Dphi}
   \begin{aligned}
   	\esssup\,(\abs{\phi(Z,W)}\st\calG)(\omega)&=\esssup\,\abs{\phi(Z_\omega,W)},\\
   	\Diam (\phi(Z,W) \st\calG) (\omega) &= \Diam \phi(Z_\omega,W).
   \end{aligned}
   \end{equation}
\end{lemma}
\begin{proof}
  Equation~\eqref{Ephi} is Theorem~6.4 in \cite{Kallenberg}.
  By applying it to functions of the form
  $\mathbf{1}_A(\phi(\cdot,\cdot))$ for each 
  $A\in\calB(\Complexes)$, we find that 
  $\phi(Z_\omega,W)\in K_\omega(\phi(Z,W)\st\calG)$.
  The $\calG$-measurability of 
  $\omega\mapsto \E\phi(Z_\omega)$ follows from the
  $\calG$-measurability of the left side of~\eqref{Ephi}. 
    Equation \eqref{Dphi} follows from \eqref{Ephi} on account  of \eqref{def_esssup} and \eqref{diam3}. Note also that 
    the $\calG$-measurability of the conditional essential supremum and the conditional diameter is just a special case of this.
\end{proof}

We now list a number of properties of the conditional diameter
that we will need.

\begin{lemma}\label{conditional}
  Let $\calG\subseteq\calF$ be a $\sigma$-field and 
  $Z$ be an a.s.\ bounded complex random variable on~$\P$.
  Then,
  \begin{itemize}\itemsep=0pt
  
  \item[(a)] $\Diam(\Re Z\st\calG) \le \Diam(Z\st\calG)$ a.s.
  
  \item[(b)] 
     $\Diam\(Z \st\calG\)
               \le 2\esssup(\abs{Z}\st\calG)$ a.s.
  
  \item[(c)] $\Diam(Z^2\st\calG)\le
   2\esssup(\abs{Z}\st\calG)\,\Diam(Z\st\calG)$ a.s.

  \item[(d)] $\abs{Z-\E(Z\st\calG)} \le \Diam(Z\st \calG)$ a.s.
  
  \item[(e)] If the $\sigma$-field $\calH\subseteq\calF$ is
    independent of~$\calG$, then
    $\Diam \E(Z\st\calH) \le \E\Diam(Z\st\calG)$ a.s.
  \item[(f)] $\Abs {\E \( (Z - \E(Z\st\calG))  (W - \E(W\st\calG)) \St \calG \)}
    \leq \tfrac{1}{3} \Diam (Z \st \calG)\cdot \Diam(W \st \calG)$ a.s.\
  for any a.s.\ bounded complex random variable $W$ on~$\P$.
  
  \item[(g)] If $U$ and $W$ are $\calG$-measurable, then
  $ \Diam(UZ+W \st \calG) = \abs{U} \,\Diam(Z\st\calG)$.
  \end{itemize}
\end{lemma}
\begin{proof}
 Let $Z_\omega$ be an arbitrary member of 
  $K_\omega(Z\st\calG)$ for each $\omega\in\varOmega$.
  By Lemma~\ref{Komega}, we have  that, for almost all $\omega\in\varOmega$,
   \begin{align*}
   	\E(Z\st\calG)(\omega)&=\E Z_\omega,\\
   	  \esssup (\abs{Z}\st \calG) (\omega) &= \esssup\, \abs{Z_\omega},\\
   	\Diam(\Re Z\st \calG) (\omega) &= \Diam \Re Z_\omega, \\ 
   	\Diam(Z\st \calG) (\omega) &= \Diam Z_\omega, \\ 
   	  \Diam(Z^2\st \calG) (\omega) &= \Diam Z_\omega^2,\\
   	\esssup(\Abs{Z-\E(Z \st \calG)}\st\calG) (\omega) 
	   &= \esssup\,\abs{Z_\omega - \E Z_\omega},\\
   	\E \( \Abs{Z - \E(Z\st\calG)}^2  \st \calG \)(\omega) 
	   &= \E \, \abs{Z_\omega - \E Z_\omega}^2 = \Var Z_\omega.
   \end{align*}
	Due to Lemma \ref{Lemma_diam}(d, e), claims (a)--(c) follow. 

In order to prove claims (d) and (e), recall from~\cite[Prop.~2.6]{Barron} that
 for a bounded real random variable $X$, 
\begin{equation}\label{esssup_Barron}
X\le\esssup(X\st\calG)\ \  \text{ a.s.}
\end{equation}
Therefore, 
\[
	\abs{Z-\E(Z \st \calG)} \leq \esssup(\abs{Z-\E(Z \st \calG)}\st\calG) \ \text{ a.s.}
\]
and we get claim (d) from Lemma \ref{Lemma_diam}(f).

Using \eqref{esssup_Barron} and the independence of $\calG$ and $\calH$,
\[
\E(X\st\calH) \le \E(\esssup(X\st\calG)\st\calH)
= \E\,\esssup(X\st\calG)  \ \text{ a.s.}
\] for a bounded real random variable $X$.  
Applying this to the
Definition~\eqref{diam3} with $X=\Re(e^{-i\theta}Z)$
and $X=-\Re(e^{-i\theta}Z)$, claim~(e) follows.

Claim (f) is due to Lemma \ref{Lemma_diam}(h) and the conditional Cauchy-Schwartz inequality
\begin{align*}
	&\Abs{\E \((Z - \E(Z\st\calG))  (W - \E(W\st\calG) ) \st \calG \)} \\
	&{\qquad}\leq \sqrt{\E \( \Abs{Z - \E(Z\st\calG)}^2  \st \calG \)} \;\sqrt{\E \( \Abs{W - \E(W\st\calG)}^2  \st \calG \)}~\text{a.s.}
\end{align*}
To prove claim (g), note that the properties of the conditional essential
supremum imply
\[  \esssup\( \Re(e^{-i\theta}(U + WZ)) \st \calG\)
  = \Re(e^{-i\theta}U) + \abs W\esssup\( \Re(e^{-i\theta+i\arg(W)}Z)\st \calG\), \]
and apply this to the definition of conditional diameter.
\end{proof}

In \cite{grandma} we proved the following generalization of a bound
of Hoeffding~\cite{Hoeffding}.

\begin{lemma}\label{Hoeffding}
If $Z$ is an a.s.~bounded complex random variable on $\P$, then
\[  \Abs{\E e^{Z-\E Z}-1} \le e^{\frac18\Diam(Z)^2}-1. \]
\end{lemma}
\begin{cor}\label{condHoeffding}
Let $Z$ be an a.s.~bounded complex random variable on $\P$
and let $\calG\subseteq\calF$ be a $\sigma$-field.
Then we have
\[  \Abs{\E(e^{Z - \E(Z\,\st\,\calG)}\st\calG)-1}
   \le e^{\frac18\Diam( Z\, \st \,\calG)^2}-1 \text{~a.s.}  \]
\end{cor}
\begin{proof}
It suffices to apply the lemma to arbitrary random variables
$Z_\omega\in K_\omega(Z\st\calG)$, with the
help of~\eqref{Dphi}.
\end{proof}


\nicebreak
\subsection{First order approach}

A sequence $\F = \calF_0,\ldots,\calF_n$ of $\sigma$-subfields
of $\calF$ is a \textit{filter} if $\calF_0\subseteq\cdots\subseteq\calF_n$.
A sequence $Z_0,\ldots,Z_n$ of random variables on $\P=(\varOmega,\calF,P)$
is a \textit{martingale with respect to $\F$} if
\begin{itemize}\itemsep=0pt
\item[(i)] $Z_j$ is $\calF_j$-measurable and has finite expectation, for $0\le j\le n$;
\item[(ii)] $\E(Z_j \st \calF_{j-1}) = Z_{j-1}$ for $1\le j\le n$.
\end{itemize}
Note that, up to almost-sure equality,
the martingale is determined by $Z_n$ and~$\F$, namely
$Z_j=\E(Z_n \st \calF_j)$ a.s.\ for each~$j$.

If $Z$ is a random variable on $\P$ and $0\le j\le n$, we use the following
notations for statistics conditional on $\calF_j$:
\begin{align*}
  \E_j Z  &= \E(Z\st\calF_j), \\
  \V_j Z &= \E\((Z-\E_j(Z))^2\St\calF_j\)
           = \EFjb{Z^2}{j} - (\EFjb{Z}{j})^2, \\
  \Diam_j Z &= \Diam(Z\st\calF_j).
\end{align*}
If $\calF_0=\{\emptyset,\varOmega\}$, which we not assume unless it is stated
explicitly, $\EFjb{Z}{0}$, $\VFjb{Z}{0}$ and $\Diam_0 Z$ equal
the unconditional versions $\E Z$, $\V Z$ and $\Diam Z$,
respectively.

An extremely large literature concerns concentration of martingales
derived from restrictions on the differences $Z_{j}-Z_{j-1}$, but most of it 
considers only real martingales and can't be assumed to hold for complex
martingales.  The fact that the real and imaginary parts of a complex 
martingale are real martingales can often be applied, but at the cost of
weaker bounds.
In any case, our aim is for estimates of the exponential rather than for
concentration.  Here again, the non-convexity of the exponential function
in the
complex plane often means that theorems and proofs for the real case
do not carry over immediately to the complex case.

\begin{thm}\label{emartin}
Let $\Z=Z_0,Z_1,\ldots,Z_n$ be an a.s.~bounded
complex-valued martingale with respect to a filter $\calF_0,\ldots,\calF_n$.
Define
\begin{equation}\label{Zxx}
R_k = \Diam_{k-1} Z_k
\end{equation}
for $1\le k\le n$.
Then 
\[
  \E_0 e^{Z_n}   = e^{Z_0}(1 + K(\Z)),
\]
where $K(\Z)$ is an $\calF_0$-measurable random variable with
\[
  \abs{K(\Z)}  \le \esssup\( e^{\frac18\sum_{k=1}^n R_k^2} \St \calF_0 \) 
      - 1 \ \text{ a.s.}
\]
\end{thm}
\begin{proof}
Since $\E_{k-1} Z_k=Z_{k-1}$, we have for $1\le k\le n$ that
\begin{align}
  \EFjb{e^{Z_k}}{k-1} &=  \EFjb{e^{Z_{k-1} + (Z_k-Z_{k-1})}}{k-1} \notag\\
         &= e^{Z_{k-1}}\,\(1 + \EFj{e^{Z_k-Z_{k-1}}-1}{k-1}\) \notag\\
         &= e^{Z_{k-1}} + U_k e^{Z_{k-1}} \label{fdiff}
\end{align}
for some $\calF_{k-1}$-measurable $U_k$ such that
$\abs{U_k}\le e^{R_k^2/8}-1$ a.s., by Corollary~\ref{condHoeffding}.

Now recall that $\abs{e^z}=e^{\Re z}$ for all $z$ and
note that $\Re Z_0,\ldots,\Re Z_n$ is also a martingale satisfying
the conditions of the theorem on account of Lemma \ref{conditional}(a).
Therefore we similarly have that
\begin{equation}\label{rfdiff}
  \EFjb{\abs{e^{Z_k}}}{k-1} = \abs{e^{Z_{k-1}}} + U'_k\,\abs{e^{Z_{k-1}}}
     = (1 + U'_k)\,\abs{e^{Z_{k-1}}}
\end{equation}
for some $\calF_{k-1}$-measurable $U'_k$ such that
$\abs{U'_k}\le e^{R_k^2/8}-1$ a.s.
Now we can prove by backwards induction
on $k$ that for $0\le k\le n$, 
\begin{equation}\label{RTPk}
    \EFjb{e^{Z_n}}{k} = e^{Z_k} + W_k e^{Z_k},
\end{equation}
where $W_k$ is $\calF_k$-measurable and
$\abs{W_k}\le \esssup\(e^{\frac18\sum_{j=k+1}^n R_j^2}\St\calF_k\)-1$ a.s.
Obviously~\eqref{RTPk} is true for $k=n$.
Now observe from~\eqref{fdiff} and~\eqref{RTPk} that
$\EFjb{Z_n}{k-1} = e^{Z_{k-1}} + U_k e^{Z_{k-1}}
   + \EFjb{(W_ke^{Z_k})}{k-1}$ and note that
$\Abs{\EFj{W_ke^{Z_k}}{k-1}}\le \esssup\(\abs{W_k}\St\calF_{k-1}\)
 \EFjb{\abs{e^{Z_k}}}{k-1}$ a.s.
Applying~\eqref{rfdiff} to the last term and combining the error 
terms using Lemma~\ref{errorterms}, we obtain~\eqref{RTPk} for
$k-1$.
The case $k=0$ gives the theorem.
\end{proof}

\nicebreak
\subsection{Second order approach}

In the following we need two technical bounds that are
in the Appendix, Lemma~\ref{new_ineq}.
We also use the following elementary lemma.

\begin{lemma}\label{telescope}
	Let $\Z = Z_0,Z_1,\ldots, Z_n$ be a bounded complex-valued 
	martingale with respect to a filter $\calF_0,\ldots, \calF_n$.
	Then
\begin{align*}
   \EFjb{(Z_n-Z_k)^2}{k} =
   \sum_{j=k+1}^n \EFjb{(Z_j-Z_{j-1})^2}{k}
\end{align*}
for $0\le k\le n$.
\end{lemma}
\begin{proof}
 For $0\le k\le j\le \ell\le n$,
 \[
 \EFjb{(Z_\ell-Z_j)^2}{k}
 = \EFj{\EFjb{(Z_\ell-Z_j)^2}{j}}{k}
 = \EFj{\EFjb{Z_\ell^2}{j}-Z_j^2}{k}
 = \EFj{Z_\ell^2-Z_j^2}{k}.
 \]
 Therefore,
 \begin{align*}
    \EFjb{(Z_n-Z_k)^2}{k} &= \EFj{Z_n^2 - Z_k^2}{k} \\
    &= \EFjb{\((Z_n^2-Z_{n-1}^2) + (Z_{n-1}^2-Z_{n-2}^2)
      + \cdots + (Z_{k+1}^2-Z_{k}^2)\)}{k} \\
    &= \sum_{j=k+1}^n \EFjb{(Z_j-Z_{j-1})^2}{k}.\qedhere
 \end{align*}
\end{proof}

\begin{thm}\label{goodthm}
	Let $\Z = Z_0,Z_1,\ldots, Z_n$ be a bounded complex-valued 
	martingale with respect to a filter $\calF_0,\ldots, \calF_n$.
	For $1\leq k \leq n$, define
	\begin{subequations}\label{T2cond}
		\begin{align}
			R_k &= \Diam_{k-1} Z_k, \\
                         Q_k &= \max \bigl\{\Diam_{k-1}\E_k (Z_n - Z_{k})^2 ,
                		 \Diam_{k-1} \E_k(\Re Z_n - \Re Z_{k})^2 \,\bigr\}.
		\end{align}
	\end{subequations}
	Then
	\begin{align*}
		\E_0 e^{Z_n} &= e^{Z_0 + \frac{1}{2} \V_0 Z_n}
		 + L(\Z) e^{\Re Z_0 + \frac{1}{2} \V_0(\Re Z_n)} \\
		 &= e^{Z_0 + \frac{1}{2} \V_0 Z_n} 
		 \(1 + L'(\Z)e^{\frac{1}{2} \V_0 (\Im Z_n)}\),
	\end{align*} 
		where $L(\Z)$, $L'(\Z)$ are 
	$\calF_0$-measurable random variables with 
	\[
		\abs{L(\Z)} = \abs{L'(\Z)} \leq
	\esssup \Bigl(
		\exp \Bigl( \,\sum_{k=1}^n 
		\(\tfrac{1}{6} R_k^3 + \tfrac{1}{6}R_kQ_k +
		\tfrac{5}{8} R_k^4 + 
	 \tfrac{5}{32} Q_k^2\)
		\Bigr)\Bigm|\calF_0\Bigr) -1 \ \text{ a.s.}
	\]
\end{thm}

\begin{proof}
All equalities and inequalities in the proof should be taken ``almost surely''.
For $1\le k\le n$ we have, using Lemma~\ref{telescope},
\begin{equation}\label{eqk}
  \begin{aligned}
		\EFjb{e^{Z_k+\frac12\VFjb{Z_n}{k}}}{k-1}  &=e^{Z_{k-1}+\frac12\VFjb{Z_n}{k-1}} \\ 
		&{\quad}+ 
		 e^{Z_{k-1}+\frac12 \EFjb{ (Z_n-Z_k)^2}{k-1}} \EFjb{
		 \((e^{A_k} - e^{A_k^2/2} - A_k) 
		 e^{B_k}\)}{k-1}\\
		 &{\quad}+ e^{Z_{k-1}+ \frac12 \EFjb{ (Z_n-Z_k)^2}{k-1}} 
		 \EFjb{\(A_k (e^{B_k}-B_k - 1) \)}{k-1}\\
		 &{\quad}+ e^{Z_{k-1}+ \frac12 \EFjb{ (Z_n-Z_k)^2}{k-1}} 
		 		 \EFjb{(A_k + A_k B_k)}{k-1} \\
		 	&{\quad}+  e^{Z_{k-1}+ \frac12\VFjb{Z_n}{k-1}} \EFjb{(e^{C_k} -1)}{k-1},
	\end{aligned}
\end{equation}
where
\begin{align*}
	A_k &= Z_k - Z_{k-1},\\
	B_k &= \dfrac12 \EFjb{ (Z_n-Z_k)^2}{k} - \dfrac12 \EFjb{ (Z_n-Z_k)^2}{k-1} = 
					\dfrac12 \VFjb{Z_n}{k} - \dfrac12 \EFjb{ \VFjb{Z_n}{k}}{k-1} ,\\
	C_k 	 
	&= \dfrac12 (Z_k - Z_{k-1})^2 + \dfrac12 \VFjb{Z_n}{k} - \dfrac12 \VFjb{Z_n}{k-1} \\
	&=B_k + \dfrac12 A_k^2 - \dfrac12 \E_{k-1} A_k^2.
\end{align*}
Note that 
\begin{align*}
   \EFjb{A_k}{k-1} =   \EFjb{B_k}{k-1} = \EFjb{C_k}{k-1} = 0.
 \end{align*}
Therefore, by the conditions of the theorem and
Lemma~\ref{conditional}(b,d),
\begin{align*}
    \abs{A_k} &\le R_k,\\
    \abs{B_k} &\le \Diam_{k-1} B_k = \tfrac12\Diam_{k-1}\E_k (Z_n - Z_{k})^2
     \leq \tfrac12 Q_k \text{~~and}     \\
    \abs{C_k} &\le \Diam_{k-1}{C_k} \le  
    \tfrac12 \Diam_{k-1}  (Z_k - Z_{k-1})^2 + \tfrac12 \Diam_{k-1}\E_k (Z_n - Z_{k})^2  \\   
     &\leq  \esssup\, \(\abs{Z_k - Z_{k-1}}^2 \St \calF_{k-1}\) + \tfrac12 Q_k 
     \leq R_k^2 + \tfrac12 Q_k\,
 \end{align*}
By Corollary~\ref{condHoeffding},
\[
  \Abs{\EFjb{(e^{C_k} -1)}{k-1}} \leq e^{\frac{1}{8}
     (R_k^2 + Q_k/2)^2} -1
	\leq e^{\frac{1}{4} R_k^4 + \frac{1}{16}Q_k^2} -1.
\]
Using Lemma~\ref{new_ineq} and Corollary~\ref{condHoeffding}
with the triangle inequality, we get that
	\begin{align*}
	\Abs{ \EFjb{
		 &\((e^{A_k} - e^{A_k^2/2} - A_k ) 
		 e^{B_k}\)}{k-1} }	
	  \leq (e^{\frac{1}{6} R_k^3+\frac{1}{8} R_k^4} -1) \EFjb{(\abs{e^{B_k}})}{k-1} \\
	  &= (e^{\frac{1}{6} R_k^3+\frac{1}{8} R_k^4}  -1)\( \EFjb{(e^{\Re B_k}-1)}{k-1} +1\)
	  \leq e^{\frac{1}{6} R_k^3+\frac{1}{8} R_k^4+ \frac{1}{32} Q_k^2} -e^{\frac{1}{32} Q_k^2}, \\
		\Abs{\EFjb{&\(A_k  (e^{B_k}-B_k-1)}{k-1}\)} \leq 
		e^{\tfrac{1}{32}  Q_k^2} +  e^{\tfrac16  R_k  Q_k + \tfrac{1}{16}  Q_k^2 + \tfrac14  R_k^4} - \tfrac{1}{6}  R_k  Q_k - 2.
	\end{align*}
By Lemma \ref{conditional}(f), we have 
$$
	\abs{\E_{k-1} A_k B_k} \leq \tfrac13 \Diam_{k-1} A_k \cdot\Diam_{k-1} B_k \leq \tfrac{1}{6}  R_k  Q_k.
$$	
Therefore, for each $k$, formula \eqref{eqk} gives
\begin{equation}\label{Ekzv}
	\begin{aligned}
	\EFjb{e^{Z_k+\frac12\VFjb{Z_n}{k}}}{k-1} =  e^{Z_{k-1}+\frac12\VFjb{Z_n}{k-1}}
	+ L_k e^{Z_{k-1}+\frac12\VFjb{Z_n}{k-1}} \\
	+ L_k' e^{Z_{k-1}+\frac12 \EFjb{ (Z_n-Z_k)^2}{k-1}}
	\end{aligned}
\end{equation}
for some $\calF_{k-1}$-measurable random variables $L_k$ and $L_k'$ with
 \begin{align*}
 	&\abs{L_k} \le e^{\frac{1}{4}  R_k^4 + \frac{1}{16} Q_k^2} -1,\\
 	&\abs{L_k'} \leq  	e^{\frac16 R_k^3+ \frac{3}{8} R_k^4 + \frac{1}{6} R_k  Q_k + \frac{3}{32} Q_k^2 } -1.
 \end{align*}
 
Now consider the martingale $X_0,\ldots,X_n$ of the real parts
of $Z_0,\ldots,Z_n$.
In order to bound the second and third terms of~\eqref{Ekzv} we
consider the absolute value
\begin{equation}\label{abs_value}
\begin{aligned}
   \abs{ e^{Z_{k-1}+\frac12\VFjb{Z_n}{k-1}} }
   &= e^{\Re Z_{k-1}+\frac12\VFjb{\Re Z_n}{k-1}
      -\frac12\VFjb{\Im Z_n}{k-1}} \le e^{X_{k-1}+\frac12\VFjb{X_n}{k-1}},\\
    \abs{e^{Z_{k-1}+\frac12 \EFjb{ (Z_n-Z_k)^2}{k-1}}}
    &=  e^{\Re Z_{k-1}+\frac12\EFjb{(\Re Z_n - \Re Z_k)^2}{k-1}
      -\frac12\EFjb{(\Im Z_n - \Im Z_k)^2}{k-1}} \\
      &\le 
      e^{X_{k-1}+\frac12\EFjb{(X_n- X_k)^2}{k-1}} \\
      &=  e^{X_{k-1}+\frac12\VFjb{X_n}{k-1}- \frac12\EFjb{(X_k- X_{k-1})^2}{k-1}}\\
       &\leq e^{X_{k-1}+ \frac12\VFjb{X_n}{k-1}}.
\end{aligned}
\end{equation}
Due to Lemma~\ref{conditional}(a), this martingale also satisfies
conditions~\eqref{T2cond}.
Therefore, by the same reasoning as before and using the inequality
\begin{equation*}
	e^{X_{k-1}+\frac12\EFjb{(X_n- X_k)^2}{k-1}} \leq e^{X_{k-1}+ \frac12\VFjb{X_n}{k-1}},
\end{equation*}
we get that
 \begin{equation}\label{EkzvR}
   \EFjb{e^{X_k+\frac12\VFjb{X_n}{k}}}{k-1}
   = e^{X_{k-1}+\frac12\VFjb{X_n}{k-1}}
      + L_k'' e^{X_{k-1}+\frac12\VFjb{X_n}{k-1}},
 \end{equation}
where 
$\abs{L_k''} \le e^{\frac{1}{4}  R_k^4 + \frac{1}{16} Q_k^2} -1
+ e^{\frac16 R_k^3  + \frac{3}{8} R_k^4 + \frac{1}{6} R_k  Q_k 
+ \frac{3}{32} Q_k^2} -1 \leq e^{\frac16 R_k^3 + \frac{5}{8} R_k^4 
+ \frac{1}{6} R_k  Q_k + \frac{5}{32} Q_k^2 } -1$.

Using~\eqref{Ekzv} and~\eqref{EkzvR}, we now prove by backwards induction on
$k$  that
\begin{equation}\label{induct}
   \EFjb{e^{Z_n}}{k} = e^{Z_k + \frac12\VFjb{Z_n}{k}}
     + M_k \,e^{X_k+\frac12\VFjb{X_n}{k}},
\end{equation}
where 
\[
\abs{M_k}\le \esssup\(
 e^{\sum_{j=k+1}^n(\frac16 R_j^3 + \frac{5}{8} R_j^4 + \frac{1}{6} R_j  Q_j + \frac{5}{32} Q_j^2)} \St \calF_k\) - 1.
\]
The claim is obviously true for $k=n$.  To perform the induction step, take
the expectation of~\eqref{induct} with respect to $\calF_{k-1}$, using
\eqref{Ekzv} and \eqref{abs_value} for the first term on the right, and~\eqref{EkzvR} to bound
the second term on the right.  Using the bound
\hbox{$e^{\frac16 R_k^3 + \frac{5}{8} R_k^4 + \frac{1}{6} R_k  Q_k 
+ \frac{5}{32} Q_k^2}-1$}
for both $\abs{L_k} + \abs{L_k'}$ and $\abs{L_k''}$, we obtain~\eqref{induct} for $\EFjb{e^{Z_n}}{k-1}$
on combining the error terms using Lemma~\ref{errorterms}.
After $n$ steps we reach the expression for $\EFjb{e^{Z_n}}{0}$ stated
in the theorem.
\end{proof}

\begin{remark}\label{diamremark}
 Although the two options on the right side of~(\ref{T2cond}b) are the
 same for real martingales, either one of them can be the largest for
 complex martingales.  The case of a purely imaginary martingale shows
 that the first can be larger.  To show that the second can be larger, 
 consider independent variables $X,Y$, where $X\in\{1,e^{2\pi i/3},e^{-2\pi i/3}\}$
 with equal probabilities, and $Y\in\{0,1\}$ with equal probabilities.
 Now consider the martingale $Z_0,Z_1,Z_2$ where
 $Z_2=XY$,
 $Z_1=\E(Z_2\st \calF_1)=0$ and $Z_0=\E(Z_1\st\calF_0) = 0$, where $\calF_0=\{\emptyset,\varOmega\}$ and $\calF_1=\sigma(Y)$.
  We find that $\E_1 (Z_2-Z_1)^2=0$ and
 $\E_1 (\Re Z_2-\Re Z_1)^2 \in\{\frac12,1\}$ with
 probabilities $\frac23,\frac13$ respectively.
 Therefore
 $\Diam_0\(\E_1 (Z_2-Z_1)^2\)=0<\frac12=
 \Diam_0\(\E_1 (\Re Z_2-\Re Z_1)^2\)$.
\end{remark}

\nicebreak
\section{Functions of independent random variables}\label{S:fXY}

In this section we apply our martingale theorems to the
case of functions of independent random variables.

An important example of a martingale is made by the so-called
\textit{Doob martingale process}.
Suppose $X_1,X_2,\ldots,X_n$ are random variables and
$f(X_1,X_2,\ldots,X_n)$ is a random variable of bounded expectation.
Then we have the martingale 
$\{Z_j\}$ with respect to $\{\calF_j\}$, where for each~$j$,
$\calF_j=\sigma(X_1,\ldots,X_j)$ and
$Z_j=\E(f(X_1,X_2,\ldots,X_n)\st \calF_j)$.
In particular, $\calF_0=\{\emptyset,\varOmega\}$ and
$Z_0=\E f(X_1,X_2,\ldots,X_n)$.
We will also use the $\sigma$-fields
$\calF^{(j)}=\sigma(X_1,\ldots,X_{j-1},X_{j+1},\ldots,X_n)$
for $1\le j\le n$.

In this section we use the Doob martingale to find bounds 
on $\E e^f$.  
 We first need some preliminary lemmas in order to show that all
 assumptions of Theorems~\ref{emartin} and~\ref{goodthm} are satisfied.

\begin{lemma}\label{condiam}
  Suppose that $X,Y$ are independent random variables
  on $\P$, and
  that $g$ is a complex-valued function such that
  $g(X,Y)$ is bounded and measurable. Then,
  \begin{itemize}\itemsep=0pt
    \item[(a)] $\Diam(g(X,Y)\st\sigma(X))(\omega)
       =\Diam g(X,Y(\omega))$
       for almost all $\omega\in\varOmega$.
    \item[(b)] $\Diam(g(X,Y)\st\sigma(X))
      \le \sup_{x\in X(\varOmega),
             \,y,y'\in Y(\varOmega)}\,\abs{g(x,y)-g(x,y')}$ a.s.
    \item[(c)] 
     $\Diam\(g(X,Y)-\E (g(X,Y)\st\sigma(Y))\St\sigma(X)\) \\
        {\quad} \le \sup_{x,x'\in X(\varOmega),
           \,y,y'\in Y(\varOmega)}\,
           \abs{g(x,y)-g(x',y)-g(x,y')+g(x',y')}$ a.s.
  \end{itemize}
\end{lemma}

\begin{proof}
 Since $Y$ is by definition $\sigma(Y)$-measurable,
 Lemma~\ref{Komega} tells us that 
 $g(X(\omega),Y)\in K_\omega(g(X,Y)\st\sigma(X))$.
 Claim~(a) is thus just the definition of the conditional
 diameter.
 Similarly, $g(X(\omega),Y)\in K_\omega(g(X,Y)\st\sigma(X))$,
 which gives claim~(b) when we apply the definition of
 $\Diam(g(X(\omega),Y))$.
 
 For claim~(c), note that for almost all $\omega\in\varOmega$,
 $\E(g(X,Y)\st\sigma(Y))(\omega)
 =\E g(X,Y(\omega))$.
  Therefore, applying claim~(b),
  \begin{align*}
    \Diam\(g(X,Y)&-\E (g(X,Y)\st\sigma(Y))\St\sigma(X)\) \\
    &{}\le \sup_{x\in X(\varOmega),\,
       y,y'\in Y(\varOmega)}
     \abs{g(x,y)-\E(X,y)-g(x,y')+\E(X,y')} \\
     &{}\le \sup_{x,x'\in X(\varOmega),\,
           y,y'\in Y(\varOmega)}
         \abs{g(x,y)-g(x',y)-g(x,y')+g(x',y')}
  \end{align*}
 since $\abs{\E U}\le \sup\,\abs{U}$ for any complex random
 variable.
\end{proof}

We will deal with functions with
additional arguments.
For these purposes we state the following corollary of  Lemma \ref{conditional} and Lemma \ref{condiam}.

\begin{cor}\label{condiam2}
  Suppose that $W,X,Y$ are independent random variables
  on $\P$, and
  that $h$ is a complex-valued function such that
  $h(W,X,Y)$ is bounded and measurable. Then, a.s.,
  \begin{itemize}\itemsep=0pt
  		\item[(a)] 
  		$
  		\Diam \( \E(h(W,X,Y) \st \sigma(W,X)) \St \sigma(W)\)
  		\leq \E \( \Diam(h(W,X,Y) \st \sigma(W,Y)) \St \sigma(W)\).
  		$ 
      \item[(b)] 
     $\Diam\(h(W,X,Y)-\E (h(W,X,Y)\st\sigma(W,Y))\St\sigma(W,X)\) \\
        {\quad} \le \sup_{w\in W(\varOmega), x,x'\in X(\varOmega),
           \,y,y'\in Y(\varOmega)}\,
           \abs{h(w,x,y)-h(w,x',y)-h(w,x,y')+h(w,x',y')}$.
  \end{itemize}
\end{cor}
\begin{proof}
Using Lemma \ref{condiam}(a), we note that  for almost all $\omega\in\varOmega$,
\begin{align*}
  \E (h(W,X,Y)\st\sigma(W,X))(\omega) &= \E (h(W(\omega),X(\omega),Y)) \\ 
  &=  \E (h(W(\omega),X,Y)\st\sigma(X)) (\omega), \\
  \Diam (h(W,X,Y)\st\sigma(W,X))(\omega) &= \Diam (h(W(\omega),X(\omega),Y))\\
       &= \Diam (h(W(\omega),X,Y)\st\sigma(X))(\omega),
\end{align*}
and the same with $X$ and $Y$ interchanged,
so we can apply Lemma \ref{conditional}(e) and Lemma~\ref{condiam}(c) 
to random variables given by the two-argument
functions $g_\omega(X,Y)=h(W(\omega),X,Y)$ to obtain both claims.
\end{proof}

\nicebreak
\subsection{Estimating the exponential}

In this section we state our main results when applied to the
case of complex functions of independent random variables.
Let $\P=(\varOmega,\calF,P)$ be a probability space.
Let $\S=S_1\times \cdots\times S_n$ be any $n$-dimensional
domain and consider a function $F:\S\to\Complexes$.
For $1\le k\le n$, define
\begin{equation}\label{fxx}
  \alpha_k (F,\S) = \sup \,\abs{F(\xvec^k)-F(\xvec)},
\end{equation}
where the supremum is over $\xvec,\xvec^k\in \S$ that differ
only in the $k$-th coordinate.
Similarly, for $j\ne k$, define
 \begin{equation}\label{4point}
  \varDelta_{jk}(F,\S) = \sup\, \abs{F(\x)-F(\x^{j})-F(\x^{k})+F(\x^{jk})},
 \end{equation}
where the supremum is over $\x,\x^k, \x^j,\x^{jk} \in \S$ such that $\x,\x^{k}$
differ only in the $k$-th component, $\x,\x^{j}$ differ only in the $j$-th
component, $\x^{j},\x^{jk}$ only in the
$k$-th component, and $\x^{k},\x^{jk}$ only in the $j$-th component.
We also define the column vector $\alphavec(F,\S)
=(\alpha_1(F,\S),\ldots,\alpha_n(F,\S))\trans$ and the matrix
$\varDelta(F,\S) = \(\varDelta_{jk}(F,\S)\)$ with zero diagonal.

\begin{thm}\label{efmartin}
Let $\X=(X_1,\ldots,X_n)$ be a random vector
 on $\P$ with independent components, and
let $f:\X(\varOmega)\to \Complexes$
be a measurable function.
Let $\alphavec=\alphavec(f,\X(\varOmega))$
and $\varDelta=\varDelta(f,\X(\varOmega))$.
\begin{itemize}\itemsep=0pt
\item[(a)] We have
\begin{equation}\label{2point_result}
  \E e^{f(\X)} = e^{\E f(\X)}(1+ K),
\end{equation}
where $K = K(f(\X))$ is a complex constant with
$  \abs{K} \le e^{\frac18\alphavec\trans\!\alphavec} - 1$.

\item[(b)] 
We have
\begin{equation}\label{4point_result}
	\begin{aligned}
   \E e^{f(\X)} &=  e^{\E f(\X) + \frac12\V f(\X)} + L\, e^{\E \Re f(\X)
     + \frac12\Var \Re f(\X)}
   \\&=
   e^{\E f(\X) + \frac12\V f(\X)}
       \(1 + L'\, e^{\frac12\Var \Im f(\X)}),
   \end{aligned}
\end{equation}
where $L= L(f(\X))$ and $L'= L'(f(\X))$ are complex constants with
\[
   \abs{L} = \abs{L'} \le \exp\biggl(
       \dfrac16\sum_{k=1}^n\alpha_k^3 
       + \dfrac16 \alphavec\trans\!\varDelta\alphavec
       + \dfrac58 \sum_{k=1}^n\alpha_k^4
       + \dfrac{5}{16} \alphavec\trans\!\varDelta^2\alphavec
   \biggr) - 1.
\]
\end{itemize}
\end{thm}

\begin{proof}
Consider the martingale $\{Z_k\}$ with respect to
$\{\calF_k\}$ obtained by the Doob martingale process.
For $1\le k\le n$ we have
\begin{align*}
  \Diam_{k-1} Z_k &= \Diam\( \E(f(\X)\st\calF_k) \St \calF_{k-1}\) \\
       &\le \E\(\,\Diam(f(\X)\st\calF^{(k)}) \St \calF_{k-1}\),
       \text{~~by Corollary~\ref{condiam2}(a)} \\
       &\leq \esssup \(\Diam(f(\X)\St\calF^{(k)})\)\\
    &\le \sup_{\x,\x^k} \, \abs{f(\x)-f(\x^k)},
       \text{~~by Lemma~\ref{condiam}(b)} \\[-0.5ex]
    &\le \alpha_k, \text{~~by assumption.}
\end{align*}
Now formula \eqref{2point_result} follows from Theorem~\ref{emartin}.

We next consider $\E_k (Z_n-Z_k)^2$, which by
Lemma~\ref{telescope} is equal to
$\sum_{j=k+1}^n \E_k (Z_j-Z_{j-1})^2$.
\begin{equation}\label{eq1}
	\begin{aligned}
 \Diam_{k-1} & \E_k (Z_j-Z_{j-1})^2 \\
 &= \Diam\( \E \( (Z_j-Z_{j-1})^2 \st \calF_{k} \) \St \calF_{k-1}\),
      \\
     &\leq \E\( \Diam \( (Z_j-Z_{j-1})^2 \st \calF^{(k)} \cap \calF_j   \) \St \calF_{k-1}\)  \text{~~by Corollary~\ref{condiam2}(a).}
     \\
 &\le 2 \esssup\,\abs{Z_j-Z_{j-1}} \cdot \esssup\;
     \Diam(Z_j-Z_{j-1}\st \calF^{(k)} \cap \calF_j ), \\
    &\hspace{8cm} \text{~~by Lemma~\ref{conditional}(c).} 
	\end{aligned}
\end{equation}
By Lemma~\ref{conditional}(d),
\begin{equation}\label{eq2}
    \abs{Z_j-Z_{j-1}} \le \Diam_{j-1} Z_j \le \alpha_j.
\end{equation}
Using Corollary~\ref{condiam2}(a,b), we find that
\begin{align}
   \Diam(Z_j-Z_{j-1}\st \calF^{(k)}\cap\calF_j)
     &= \Diam\( \E_j(f(\X)-\E(f(X)\st\calF^{(j)}) ) \st \calF^{(k)} \cap \calF_j \) \notag\\
     &\le \E\( \Diam\(f(\X)-\E(f(X)\st\calF^{(j)}) \st \calF^{(k)}\) \st   \calF^{(k)} \cap \calF_j\)  \notag\\ 
     &\leq  \esssup\,  \Diam\(f(\X)-\E(f(X)\st\calF^{(j)}) \St \calF^{(k)}\)  \notag\\
     &\leq \sup\, \abs{f(\x)-f(\x^{j})-f(\x^{k})+f(\x^{jk})} \notag\\
     &\le\varDelta_{jk}, \text{~~by assumption}, \label{eq3}
\end{align}
where the last supremum is over  $\x, \x^k, \x^{j},\x^{jk}\in\X(\varOmega)$ such that
 $\x,\x^{k}$ differ only in the $k$-th component, $\x,\x^{j}$ differ only in the $j$-th component, $\x^{j},\x^{jk}$ only in the
$k$-th component, and $\x^{k},\x^{jk}$ only in the $j$-th component.
Combining~\eqref{eq1}--\eqref{eq3}, we obtain that
\[
    \Diam_{k-1}  \E_k (Z_j-Z_{j-1})^2 \le 2\alpha_j\varDelta_{jk}.
\]

The same bound holds for
$\Diam_{k-1}\E_k(\Re Z_n-\Re Z_k)^2$, since
the Doob martingale of $\Re f(\X)$ also satisfies conditions
(a) and (b) of the theorem.

Now we can apply Theorem~\ref{goodthm} to obtain  \eqref{2point_result} with
\[
 \abs{L(f,\X)} \le \exp\biggl(
     \dfrac16\sum_{k=1}^n\alpha_k^3
     + \dfrac{5}{8}\sum_{k=1}^n\alpha_k^4
     + \dfrac13\sum_{k=1}^n\sum_{j=k+1}^n\alpha_j\alpha_k\varDelta_{jk}
     + \dfrac58\sum_{k=1}^n
          \Bigl(\sum_{j=k+1}^n\alpha_j\varDelta_{jk}\Bigr)^{\!2}
     \biggr)-1.
\]
Since the matrix $\varDelta$ is symmetric, the third term in the
summation equals
$\tfrac16\alphavec\trans\!\varDelta\alphavec$.

The term $A=\sum_{k=1}^n \(\sum_{j=k+1}^n\alpha_j\varDelta_{jk}\)^2$
depends on the order that the arguments of $f$ are listed, but we
can define the martingale using any order we wish.
If we write $A=\sum_{j>k,\ell>k}\alpha_j\varDelta_{jk}\varDelta_{k\ell}\alpha_\ell$,
then the version from the reverse order of the arguments is
$A'=\sum_{j<k,\ell<k}\alpha_j\varDelta_{jk}\varDelta_{k\ell}\alpha_\ell$.
Since $A$ and $A'$ provide disjoint sets of terms of
$\alphavec\trans\!\varDelta^2\alphavec
= \sum_{j,k,\ell} \alpha_j\varDelta_{jk}\varDelta_{k\ell}\alpha_\ell$,
at least one of them is bounded by $\tfrac12\alphavec\trans\!\varDelta^2\alphavec$.
This completes the proof.
\end{proof}

\begin{remark}\label{Catoni-remark}
A result similar to Theorem~\ref{efmartin} was proved by
Catoni~\cite{Catoni1,Catoni2} when the function $f$ is real,
and used to obtain concentration bounds of the form
\[
 P\(f(\X)\ge \E f(\X) + t\) \le 
 \exp\biggl(-\frac{t^2}{2\(\Var f(\X)+\eta t/\Var f(\X)\)}\biggr),
\]
where $\eta$ is a certain constant depending on $\alphavec$
and $\varDelta$. We won't pursue that direction here since
we are interested in the complex case which is
required for our applications.  The complex case
has the added advantage that we can use it to estimate characteristic
functions and not just Laplace transforms, with interesting
consequences that include Berry--Esseen-type inequalities which
we will explore in a further paper.

Another point to mention in comparison with Catoni's theorems
is that he doesn't have fourth-order terms such as the term
$\frac 58\sum_{k=1}^n\alpha_k^4$ in Theorem~\ref{efmartin}(b).
Although those terms make the bound much larger for very
large $\{\alpha_j\}$, in such extreme cases part~(a) of the theorem
generally gives a better result anyway. We have included fourth
order terms in order to allow better constants on the
third order terms.
\end{remark}

\begin{remark}
The factor $e^{\frac12\Var\Im f(\X)}$ appearing
in the error term of~\eqref{4point_result} is of course redundant in
the case that~$f$ is real.  The following example shows that some such
multiplier is required in the general complex case.
Suppose that the components of $\X=(X_1,\ldots,X_n)$ are iid random
variables with mass~$\frac12$ at each of $\pm n^{-1/2+\eps}$.
Define $X=\sum_{j=1}^n X_j$ and $f(\X)=iX+\frac1n e^{-iX}$.
We obviously have $\E X=0$ and $\E X^2=n^{2\eps}$.
For $c=\pm 1$, we have
\begin{align*}
  \E e^{icX} &= \( \E e^{icX_1} \)^n
   = \( \tfrac12 e^{-in^{-1/2+\eps}}+ \tfrac12 e^{in^{-1/2+\eps}} \)^n \\
   &= \( 1 - \tfrac12 n^{-1+2\eps} + O(n^{-3/2+3\eps}) \)^n
     = e^{-n^{2\eps}/2+O(n^{-1/2+3\eps})}.
\end{align*}
Using $e^{f(\X)}=e^{iX}+\tfrac1n+O(\tfrac{1}{2n^2})$ we have
$\E e^{f(\X)}=\tfrac{1}{n}+O(\tfrac{1}{n^2})$.  Now let us apply
Theorem~\ref{efmartin}. We have $\E f(\X)= \frac1n e^{-n^{2\eps}/2+o(1))}$
and $\V f(\X)=-n^{2\eps}+o(1)$.  Therefore
$e^{\E f+\frac12\V f}=e^{-\frac12n^{2\eps}+o(1)}$.
In the error term of~\eqref{4point_result} we have
$\alpha_k=O(n^{-1/2+\eps})$ and $\varDelta_{jk}=O(n^{-2+2\eps})$.
So $e^{\E f+\frac12\V f}$ is very much smaller than 
$\E e^{f(\X)}$ even though $L'=o(1)$.  In a later paper we will
investigate a wide class of complex functions for which a 
theorem similar to Theorem~\ref{efmartin} is true without the
factor~$e^{\frac12\Var\Im f(\X)}$.
\end{remark}

\subsection{Smooth and transformed functions}\label{S:smooth}

In the case of smooth functions, the parameters $\alpha_j$ and
$\varDelta_{jk}$ can be bounded in terms of derivatives or other
measures such as Lipschitz constants.  For our applications in 
Section~\ref{S:gauss}, it will suffice to have continuous differentiability.

If $f$ is a function one of whose arguments is $x$, then $f_x$ is the
partial derivative $\partial f/\partial x$, and similarly for notations like
$f_{xy}$.  If the arguments are a subscripted list, like $x_1,\ldots,x_n$,
we will further abbreviate $f_{x_j}$ to $f_j$ and $f_{x_jx_k}$ to~$f_{jk}$.
The notations $\norm{\cdot}_1$, $\norm{\cdot}_2$ and $\norm{\cdot}_\infty$
have their usual meanings as vector norms and the corresponding
induced matrix norms.  For a matrix $A=(a_{jk})$ we will also use
$\maxnorm{A}=\max_{jk}\,\abs{a_{jk}}$ but note that it is
not submultiplicative.

\begin{lemma}\label{simpleints}
~\\[-4ex]   
\begin{itemize}
  \item[(a)]
  Let $L$ be the closed line segment $[x_1,x_2]\subseteq\Reals$
  and let $S$ be its interior minus a countable set of points. 
  Suppose that the function
  $f:L\to\Complexes$ is continuous, and that
  $f_x$ exists and is bounded in~$S$. Then
  \[
    \abs{f(x_2)-f(x_1)} \le
      \abs{x_2-x_1} \,\sup_{x\in S} \,\abs{f_x(x)}.
  \]
  \item[(b)]
   Let $R$ be the closed rectangle 
   $[x_1,x_2]\times[y_1,y_2]\subseteq\Reals^2$ and let
   $S$ be its interior minus a countable set of lines.
   Suppose that the function
   $f:R\to\Complexes$ is continuous
   and $f_x$ exists and is continuous.  Moreover assume that
   $f_{xy}$ exists and is bounded in $S$.  Then
   \[
      \Abs{f(x_1,y_1)-f(x_1,y_2)-f(x_2,y_1)+f(x_2,y_2)}
        \le \abs{x_2-x_1}\,\abs{y_2-y_1} \sup_{(x,y)\in S} \abs{f_{xy}(x,y)}.
   \]
\end{itemize}
\end{lemma}
\begin{proof}
 The conditions we have given are sufficient to imply that
 \[
    f(x_2)-f(x_1)
     = \int^{\scriptscriptstyle\mathrm{(HK)}}_L f_x(x)\,dx
 \]
 in case (a) and
 \[
  f(x_1,y_1)-f(x_1,y_2)-f(x_2,y_1)+f(x_2,y_2)
  = \int^{\scriptscriptstyle\mathrm{(HK)}}_{[y_1,y_2]}\biggl(\,
     \int^{\scriptscriptstyle\mathrm{(HK)}}_{[x_1,x_2]} f_{xy}(x,y)\,dx
     \biggr)\,dy,
 \]
 in case (b), where we have used the
 Henstock--Kurzweil (gauge) integral~\cite[Thm.~4.7]{Bartle}.
 The claims now follow readily.
\end{proof}

Note that in part (a) we did not require that $f_x$ is continuous, and
in part (b) we did not require that $f_y$ or $f_{xy}$ are continuous.
The lemma is not true if ``countable'' is replaced by ``measure zero''.
In the following we will adopt more stringent conditions on derivatives
than Lemma~\ref{simpleints} allows, leaving the generalizations to future applications.

\begin{cor}\label{derivs}
 Let $B=[a_1,b_1]\times\cdots\times[a_n,b_n]\subseteq\Reals^n$.
 Suppose that $f: B \to \Complexes$ is continuous.
 Then, provided the suprema exist,
 \begin{itemize}\itemsep=0pt
  \item[(a)] 
  If $f$ is continuously differentiable in the interior $\interior B$
  of $B$, then, for $1\le j\le n$,
   \[
       \alpha_j(f,B)  \le (b_j-a_j) \sup_{\xvec\in \interior B}
           \abs{f_j(\xvec)}.
   \]
  \item[(b)] If $f$ is twice continuously differentiable
   in~$\interior B$,
   then, for $1\le j<k\le n$,
    \[
         \varDelta_{jk}(f,B) \le (b_j-a_j)(b_k-a_k)
         \sup_{\xvec\in \interior B}
         \abs{f_{jk}(\xvec)}.
    \]
 \end{itemize}
\end{cor}
\begin{proof}
  This follows immediately from Lemma~\ref{simpleints},
  noting that line segments or rectangles in the boundary
  of~$B$ are limits of line segments or rectangles not in the 
  boundary.
\end{proof}

In the case of a transformed cuboid, it is convenient to be able to bound
$\norm{\alphavec}_\infty$, $\alphavec\trans\!\varDelta\alphavec$ and $\alphavec\trans\!\varDelta^2\alphavec$ in terms of the derivatives
in the original coordinates.  We will only treat the case of
uniformly bounded derivatives.
For $\rho>0$, define
\[
    U_n(\rho) = \{ \xvec\in\Reals^n \st \abs{x_j}\le\rho \text{~for~} 1\le j\le n \}.
\]

If $B\subseteq \mathbb{R}^n$ is some set and $f: B \rightarrow \Complexes$ 
is twice differentiable in some open set containing $B$, define the
matrix $H(f,B) = (h_{jk})$, where, provided the suprema exist,
$h_{jk}=\sup_{\yvec\in B} \,\abs{f_{jk}(\yvec)}$.

\begin{lemma}\label{smooth}
  For some $\rho>0$, let $B=U_n(\rho)$.
  Suppose that $T:\Reals^n\to\Reals^m$ is a differentiable transformation and 
  let $J_T$ denote its Jacobian matrix.   Let $S\subseteq\Reals^m$ be an open set that contains $T(\interior B)$.
  Suppose $f:T(B)\cup S \to\Complexes$ is continuous, and
  define $\tilde f:B\to\Complexes$ by $\tilde f(\xvec)=f(T(\xvec))$.
  Write $\alphavec=\alphavec(\tilde f,B)$ and
  $\varDelta=(\varDelta_{jk})=\varDelta(\tilde f,B)$, Then
  \begin{itemize}\itemsep=0pt
   \item[(a)] Suppose that $f$ is continuously differentiable in
   $S$ with $\abs{f_j(\yvec)} \le m_1$ for $\yvec\in T(\interior B)$
   and $1\le j\le n$.  Then
   \[
       \norm{\alphavec}_\infty \le 2\rho  \,m_1 \sup_{\x \in \interior B} \norm{J_T(\xvec)}_1.
   \]
   \item[(b)] Suppose that $f$ is twice continuously
   differentiable in $S$ with
   $\norm{H(f,T(\interior B))}_\infty \leq m_2$.  Then
  \begin{align*}
     \alphavec\trans\!\varDelta\alphavec
      &\le 4 \rho^2 n m_2\norm{\alphavec}^2_\infty \,\sup_{\xvec \in \interior B} \( \norm{J_T(\xvec)}_1\, \norm{J_T(\xvec)}_\infty\) \text{~~and}\\
     \alphavec\trans\!\varDelta^2\alphavec
           &\le 16 \rho^4 n m_2^2 \norm{\alphavec}^2_\infty \,
           \sup_{\xvec \in \interior B} \(\norm{J_T(\xvec)}_1\, \norm{J_T(\xvec)}_\infty\)^2.
  \end{align*}
  \end{itemize}
\end{lemma}
\begin{proof}
Suppose $J_T=(t_{jk}(\xvec))$. Observe that for $\x \in \interior B$ 
\begin{align*}
 \tilde{f}_j(\x) &= \sum\limits_{r=1}\limits^m t_{rj}(\x) f_r(T(\x)), \\
  \tilde{f}_{jk}(\x) &= 
  \sum\limits_{r=1} \limits^m \sum\limits_{s=1} \limits^n
   t_{rj}(\x) f_{rs}(T(\x)) t_{sk} (\x).
\end{align*}
From Corollary \ref{derivs} for function $\tilde{f}$, we get
\[
	\alpha_j \leq 2 \rho m_1  \sup_{\x \in \interior B} \;\sum_{r=1}^m \,\abs{t_{rj}(\x)},
\] 
which is equivalent to  part (a), and 
\[
	\varDelta_{jk} \leq 4 \rho^2 \sup_{\x \in \interior B}\; \sum_{r=1}^m \sum_{s=1}^n
   \abs{t_{rj}(\x)}\, h_{rs} \,\abs{t_{sk} (\x)},
\]
where $H(f,T(\interior B))= (h_{rs})$.  
Note that the expression $\sum\limits_{r=1} \limits^m \sum\limits_{s=1} \limits^n
   \abs{t_{rj}(\x)}\, h_{rs}\, \abs{t_{sk} (\x)}$ of the right hand side is the
$(j,k)$ element of $(\hat{J}_T)\trans  H(f,T(\interior B)) \hat{J}_T$, 
where $\hat{J}_T =(\abs{t_{rs}})$. Claim (b) now follows on recalling 
that the $\infty$-norm is submultiplicative.
\end{proof}

\nicebreak
\section{Truncated gaussian measures}\label{S:gauss}

In this section explore the application of Theorem~\ref{efmartin}
to the case where the distribution of $\X$ is a truncated gaussian.
This is the case that has occurred in the most applications so far.

It will often be convenient to approximate the expectation and
pseudovariance of a complex function
of a truncated gaussian by their values for the unrestricted
gaussian.  The following gives a general principle.

\begin{lemma}\label{truncated}
  Let $A$ be an $n\times n$ symmetric positive-definite real matrix.
  Let $f:\Reals^n\to\Complexes$ be a measurable function satisfying
  \begin{equation}\label{notfast}
      \abs{f(\xvec)} \le e^{\frac bn\xvec\trans\! A\xvec}
  \end{equation}
  for all $\xvec\in\Reals^n$ and some~$b\ge 0$.
  Let $\X:\Reals\to\Reals$ be a random variable with density
  \[
       \pi^{-n/2}\,\abs{A}^{1/2}\, e^{-\xvec\trans\! A\xvec}.
  \]
  Suppose $\varOmega$ is a measurable subset of
  $\Reals^n$ and define $p = \Prob(\X\notin\varOmega)$.
  Then, if $p\le\tfrac34$ and $n\ge b+b^2$, we have
  \[
     \Abs{\E (f(\X) \st \X\in\varOmega) - \E f(\X)}
     \le  15\, e^{b/2} p^{1-b/n}.
  \]
  Moreover, for $p\le\tfrac34$ and $n\ge 2b+4b^2$, we have
  \[
       \Abs{\V\, (f(\X) \st \X\in\varOmega) - \V f(\X)}
       \le  112\, e^{b} p^{1-2b/n}.
  \]
\end{lemma}
\begin{proof}
 By linear transform we can assume that $A=\tfrac12 I$ and that
 $\abs{f(\xvec)} \le e^{\frac b{2n}\xvec\trans\!\xvec}$.
 Let $\mu$ denote the measure with density
 $\pi^{-n/2} e^{-\frac12 \xvec\trans\!\xvec}$, which is the
gaussian measure defined by~$\X$ after transformation.
From the definition of expectation,
  \[
    \E (f(\X) \st \X\in\varOmega) - \E f(\X) 
       = (1-p)^{-1} \biggl( p\int_{\Reals^n} f(\xvec)\,d\mu
               - \int_{\Reals^n-\varOmega} f(\xvec)\,d\mu \biggr).
 \]
 For any $r>0$, 
 since $\Prob(\X\notin\varOmega\wedge \abs{\X}\le r) \le p$,
 we can bound
 \[
      \int_ {\Reals^n-\varOmega} \abs{f(\xvec)}\,d\mu
      \le
       p\sup_{\abs{\xvec}\le r} \,\abs{f(\x)}
         + \int_{\abs{\xvec}\ge r}  \abs{f(\x)}\,d\mu.
 \]
Consequently we have
 \[
    \Abs{\E (f(\X) \st \X\in\varOmega) - \E f(\X)} \le
    (1-p)^{-1}\biggl( pe^{\frac{b}{2n}r^2}
      + \int_{\abs{\xvec}\ge r} e^{\frac{b}{2n}\xvec\trans\!\xvec}\,d\mu
       + p  \int_{\Reals^n} e^{\frac{b}{2n}\xvec\trans\!\xvec}\,d\mu\biggr).
\]
The second integral is easily calculated to be $(1-b/n)^{-n/2}$,
provided $n>b$.  The first integral has no closed form;
it is $(1-b/n)^{-n/2}F_n((1-b/n)r^2)$, where $F_n(u)$ denotes
the upper tail of the $\chi^2$-distribution with $n$ degrees of
freedom.  From~\cite[(4.3)]{Laurent} we have that
$F_n(n+2u^{1/2}n^{1/2}+2u)\le e^{-u}$ for any $u\ge0$.
Consequently, if we put $r^2=(n+2u^{1/2}n^{1/2}+2u)/(1-b/n)$,
we find for any $u\ge 0$ and $n > b$ that
\begin{align*}
    \Abs{\E (f(\X) \st \X\in\varOmega) &- \E f(\X)} \\
    &{}\le
    (1-p)^{-1}\( pe^{b(1/2+(u/n)^{1/2}+u/n)/(1-b/n)}
       + (1-b/n)^{-n/2} (p + e^{-u}) \).
\end{align*}
To obtain the version in the theorem statement, use
\[
    u = (1-b/n)\ln (1/p) + \frac{b\sqrt{4(n-b)\ln(1/p)-b^2}}{2n},
\]
which satisfies the equation 
$b(1/2+(u/n)^{1/2}+u/n)/(1-b/n)=-u+\ln(1/p)$.
The conditions $p\le\tfrac34$ and $n\ge b+b^2$ imply that
the argument of the square root is positive.
Now note that for $n\ge b+b^2, b\ge0$ the function
$(1-b/n)^{-n/2}e^{-b/2}<e^{1/4}$ is increasing in~$b$
and nonincreasing in~$n$, so $(1-b/n)^{-n/2}e^{-b/2}<e^{1/4}$.
Applying this bound and also $u\ge (1-b/n)\ln(1/p)$
completes the proof of the first part.

For the second part, we have
\begin{align*}
   \V\,\(f(\X)&\st\X\in\varOmega\) - \V f(\X) =
   \E(f(\X)^2\st\X\in\varOmega) - \E f(\X)^2 \\
   &{}- \(\E(f(\X)\st\X\in\varOmega) - \E f(\X)\)
     \,\(\E f(\X)+\E(f(\X)\st\X\in\varOmega)\).
\end{align*}
 Note from above that $\abs{\E f(\X)}\le \E\,\abs{f(\X)}\le(1-b/n)^{-n/2}$.
 Using the definition of $p$, we have $\Abs{\E(f(\X)\st\X\in\varOmega)}
 \le (1-p)^{-1}\int_\varOmega \abs{f(\X)}\,d\mu\le 4(1-b/n)^{-n/2}$.
 Now apply the first part of the lemma to $f(\X)$ and 
 $f(\X)^2$, as well as the bound $(1-b/n)^{-n/2}e^{-b/2}<e^{1/4}$
 used earlier.
 This completes the proof.
\end{proof}

Lemma~\ref{truncated} is not useful for exponential functions
on account of condition~\eqref{notfast}.  However, since~\eqref{notfast}
is satisfied by all polynomials (after scaling),
the lemma becomes useful in
conjunction with Theorem~\ref{efmartin} for estimating $\E e^f$
when $f$ has polynomial growth.
For convenience, we give the theorem of Isserlis~\cite{Isserlis}
(see~\cite{Holmquist} for a treatment in modern notation)
that tells us how to compute the expectations of polynomials with
respect to a multivariate normal distribution.

\begin{thm}\label{moments}
  Let $A$ be a positive-definite real symmetric matrix of
  order~$n$ and
  let  $\X=(X_1,\ldots,X_n)$ be a random variable with the
  normal density $\pi^{-n/2}\abs{A}^{1/2} e^{-\xvec\trans\!A\xvec}$.
  Let $\varSigma=(\sigma_{jk})=(2A)^{-1}$ be the corresponding
  covariance matrix.
  Consider a product $Z=X_{j_1}X_{j_2}\cdots X_{j_k}$, where the 
  subscripts do not need to be distinct.  If $k$ is odd, then
  $\E Z=0$.  If $k$ is even, then
  \[
     \E Z = \sum_{(i_1,i_2),(i_2,i_3),\ldots,(i_{k{-}1},i_k)}
        \sigma_{j_{i_1}j_{i_2}}\cdots\sigma_{j_{i_{k{-}1}}j_{i_k}},
  \]
  where the sum is over all unordered partitions of $\{1,2,\ldots,k\}$ into
  $k/2$ disjoint unordered pairs. The number of terms in the sum is
  $(k-1)(k-3)\cdots3\cdot 1$.
\end{thm}

The following are examples of Theorem~\ref{moments}.
\begin{align*}
  \E X_1^2 &= \sigma_{11} 
  & \E X_1^4 &= 3\sigma_{11}^2 \\
  \E X_1^2X_2^2 &= \sigma_{11}\sigma_{22}+2\sigma_{12}^2
  & \E X_1^2X_2X_3 
       &= \sigma_{11}\sigma_{23} + 2\sigma_{12}\sigma_{13} \\
  \E X_1X_2X_3X_4 &= \sigma_{12}\sigma_{34}
       + \sigma_{13}\sigma_{24} + \sigma_{14}\sigma_{23}
   & \E X_1^6 &= 15\sigma_{11}^3 
\end{align*}

\subsection{Truncated gaussian measures of full rank}\label{S:fullrank}

\begin{thm}\label{gauss2pt}
  Let $c_1,c_2,c_3,\eps,\rho_1,\rho_2,\phi_1,\phi_2$ be
  nonnegative real constants with $c_1,\eps>0$.
 Let $A$ be an $n\times n$ positive-definite symmetric real matrix
 and let $T$ be a real matrix such that $T\trans\! AT=I$.
 Let $\varOmega$ be a measurable set such that
 $U_n(\rho_1)\subseteq T^{-1}(\varOmega)\subseteq U_n(\rho_2)$,
   and let 
   $f: \Reals^n\to\Complexes$,
   $g: \Reals^n\to\Reals$ and $h:\varOmega\to\Complexes$ 
   be measurable functions.
 We make the following assumptions.
   \begin{itemize}\itemsep=0pt
     \item[(a)] $c_1(\log n)^{1/2+\eps}\le\rho_1\le\rho_2$.
     \item[(b)] For $\xvec\in T(U_n(\rho_1))$,
        $2\rho_1\,\norm{T}_1\,\abs{f_j(\xvec)}
         \le \phi_1 n^{-1/2}$ for each~$j$.
     \item[(c)] For $\xvec\in\varOmega$, $\Re f(\xvec) \le g(\xvec)$.
        For $\xvec\in T(U_n(\rho_2))$,
       $2\rho_2\,\norm{T}_1\,\abs{g_j(\xvec)}\le 
        \phi_2 n^{-1/2}$ for each~$j$.
     \item[(d)] $\abs{f(\xvec)},\abs{g(\xvec)} \le n^{c_3} e^{c_2\xvec\trans\! A\xvec/n}$ for $\xvec\in\Reals^n$.
    \end{itemize}
  Let $\X$ be a random variable with the normal density
    $\pi^{-n/2} \abs{A}^{1/2} e^{-\xvec\trans\!A\xvec}$.
     Then, provided $\E f(\X)$ and $\E g(\X)$ are finite
     and $h$ is bounded in~$\varOmega$,
     \[
        \int_\varOmega e^{-\xvec\trans\!A\xvec + f(\xvec)+h(\xvec)}\,d\xvec
        = (1+K) \pi^{n/2}\abs{A}^{-1/2} e^{\E f(\X)},
     \]     
   where, for some constant $C$ depending only on $c_1,c_2,c_3,\eps$,
   \[
      \abs{K} \le C \( e^{\frac18\phi_1^2+e^{-\rho_1^2/2}}-1
      + (2e^{\frac18\phi_2^2+e^{-\rho_1^2/2}}-2
        + \sup_{\xvec\in\varOmega}\,\abs{e^{h(\xvec)}-1})\,
           e^{\E(g(\X)-\Re f(\X))}\).
   \]
   In particular, if $n\ge (1+c_2)^2$ and $\rho_1^2\ge 7 + 2c_2+(3+4c_3)\log n$,
    we can take~$C=1$.
\end{thm}
\begin{proof}
  We will use Lemma~\ref{errorterms} repeatedly to combine error terms.
  Change variables by $\xvec = T\yvec$.  Since $\abs{T}=\abs{A}^{-1/2}$,
  we have
  \[
     \int_\varOmega e^{-\xvec\trans\!A\xvec + f(\xvec)+h(\xvec)}\,d\xvec
     = \abs{A}^{-1/2} \int_{T^{-1}(\varOmega)} e^{-\yvec\trans\!\yvec + f(T\yvec)+h(T\yvec)}\,d\yvec.
  \]

  Suppose $\rho\ge c_1(\log n)^{1/2+\eps}$ and let
  $F:U_n(\rho)\to\Complexes$ be measurable and such that 
  $\abs{F(\xvec)}\le n^{c_3}e^{c_2\xvec\trans\!\xvec/n}$ for
  $\xvec\in\Reals^n$ and
  $\norm{\alphavec(F,U_n(\rho))}_\infty\le \phi n^{-1/2}$.
  By Theorem~\ref{efmartin}(a),
  \[
    \int_{U_n(\rho)} e^{-\yvec\trans\!\yvec+F(\yvec)}\,d\yvec = 
      (1 + K')\, e^{\E(F(\Y)\st \Y\in U_n(\rho))} 
       \int_{U_n(\rho)} e^{-\yvec\trans\!\yvec}\,d\yvec,
  \]
  where
  $\abs{K'} \le e^{\frac18\phi^2}-1$
  and $\Y$ has the normal density $\pi^{-n/2}e^{-\yvec\trans\!\yvec}$.
  
  Define $p=\Prob(\Y\notin U_n(\rho))$. By standard
  bounds on the tail of the normal distribution,
  $p\le ne^{-\rho^2}/(1+\rho)$. Under our assumptions, there
  is $n_0=n_0(c_1,c_2,c_3,\eps)$ such that for $n\ge n_0$,
  we have $p\le\frac34$, $n\ge c_2+c_2^2$ and
  $15n^{c_3}e^{c_2/2}p^{1-c_2/n} \le e^{-\rho^2/2}$.
  Those three conditions are enough that
  we can apply Lemma~\ref{truncated} to the function
  $n^{-c_3}F(\yvec)$ to conclude that
  \begin{equation}\label{newgen2}
     \int_{U_n(\rho)} e^{-\yvec\trans\!\yvec+F(\yvec)}\,d\yvec = 
           (1+K'')\, \pi^{n/2} e^{\E F(\Y)} ,
  \end{equation}
  where $\abs{K''} \le e^{\frac18\phi^2+e^{-\rho^2/2}}-1$.
  
  We can finish the proof by applying~\eqref{newgen2} to 
  each of the functions $f(T\yvec)$ and $g(T\yvec)$.
  For $n<n_0$ we can increase $C$ to make the theorem hold,
  so assume $n\ge n_0$.
  By Lemma~\ref{smooth} we have
  $\norm{\alphavec(f(T\yvec),U_n(\rho_1))}_\infty \le n^{-1/2}\phi_1$ and
  $\norm{\alphavec(g(T\yvec),U_n(\rho_2))}_\infty \le n^{-1/2}\phi_2$.
  Now we have
  \begin{align}
     \int_\varOmega e^{f(T\yvec)+h(T\yvec)-\yvec\trans\!\yvec}\,d\yvec &= 
     \int_{U_n(\rho_1)} e^{f(T\yvec)-\yvec\trans\!\yvec}\,d\yvec \notag \\
     & {\qquad}+\int_{\varOmega \setminus U_n(\rho_1)} 
     e^{f(T\yvec)-\yvec\trans\!\yvec}\,d\yvec 
     + \int_{\varOmega } (e^{h(\yvec)}-1) e^{f(T\yvec)-\yvec\trans\!\yvec}\,d\yvec\notag  \\
     &=  \int_{U_n(\rho_1)} e^{f(T\yvec)-\yvec\trans\!\yvec}\,d\yvec
       + A \biggl(\, \int_{U_n(\rho_2)}-\int_{U_n(\rho_1)}\,\biggr)
            e^{g(T\yvec)-\yvec\trans\!\yvec}\,d\yvec \notag\\
       &{\qquad}+ A' \sup_{\xvec\in\varOmega}\,
        \abs{e^{h(\xvec)}-1}\int_{\Reals^n} e^{g(T\yvec)-\yvec\trans\!\yvec}\,d\yvec
        \text{~~for $\abs{A},\abs{A'}\le 1$} \label{3parts}\\
      &= \pi^{n/2} e^{\E f(\X)}(1 + K_1) + K_2 \pi^{n/2} e^{\E g(\X)}
         + K_3\pi^{n/2} e^{\E g(\X)}, \notag
  \end{align}
  where we have
  $\abs{K_1}\le e^{\frac18\phi_1^2+e^{-\rho_1^2/2}}-1$,
  $\abs{K_2}\le 2e^{\frac18\phi_2^2+e^{-\rho_1^2/2}}-2$ and
  $\abs{K_3}\le \sup_{\xvec\in\varOmega}\, \abs{e^{h(\xvec)}-1}$.
  Finally note that $\abs{e^{\E f(\X)}} = e^{\E\Re f(\X)}$; the theorem follows.
  
  To establish the final claim, it will suffice to show that for
  $\rho^2\ge 7 + 2c_2+(3+4c_3)\log n$ we can prove~\eqref{newgen2}
  with $n_0=(1+c_2)^2$.
  Obviously $n\ge (1+c_2)^2$ implies that $n\ge c_2+c_2^2$, and
  it also implies that $1-c_2/n\ge\frac34$.
  The bounds $\rho^2\ge 7+3\log n$ and $p\le ne^{-\rho^2}/(1+\rho)$\
  imply that $p\le\frac34$ and also that $p\le e^{-\rho^2}n/(1+\sqrt 7)$.
  Combining these bounds produces the third required
  inequality $15n^{c_3}e^{c_2/2}p^{1-c_2/n} \le e^{-\rho^2/2}$,
  completing the proof.
\end{proof}

\begin{thm}\label{gauss4pt}
  Let $c_1,c_2,c_3,\eps,\rho_1,\rho_2,\phi_1,\phi_2$ be
  nonnegative real constants with $c_1,\eps>0$.
 Let $A$ be an $n\times n$ positive-definite symmetric real matrix
 and let $T$ be a real matrix such that $T\trans\! AT=I$.
 Let $\varOmega$ be a measurable set such that
 $U_n(\rho_1)\subseteq T^{-1}(\varOmega)\subseteq U_n(\rho_2)$,
   and let 
   $f: \Reals^n\to\Complexes$,
   $g: \Reals^n\to\Reals$ and $h:\varOmega\to\Complexes$ 
   be measurable functions.
 We make the following assumptions.
   \begin{itemize}\itemsep=0pt
     \item[(a)] $c_1(\log n)^{1/2+\eps}\le\rho_1\le\rho_2$.
     \item[(b)] For $\xvec\in T(U_n(\rho_1))$,
        $2\rho_1\,\norm{T}_1\,\abs{f_j(\xvec)}
         \le \phi_1 n^{-1/3}\le\tfrac23$ for $1\le j\le n$ and\\
         $4\rho_1^2\,\norm{T}_1\,\norm{T}_\infty\,
         \norm{H(f,T(U_n(\rho_1)))}_\infty
         \le \phi_1 n^{-1/3}$.
     \item[(c)] For $\xvec\in\varOmega$, $\Re f(\xvec) \le g(\xvec)$.
        For $\xvec\in T(U_n(\rho_2))$, either\\
       (i) $2\rho_2\,\norm{T}_1\,\abs{g_j(\xvec)}\le 
        (2\phi_2)^{3/2} n^{-1/2}$ for $1\le j\le n$, or\\
       (ii) $2\rho_2\,\norm{T}_1\,\abs{g_j(\xvec)}
         \le \phi_2 n^{-1/3}$ for $1\le j\le n$ and\\
         \hspace*{1.7em}$4\rho_2^2\,\norm{T}_1\,\norm{T}_\infty\,
          \norm{H(g,T(U_n(\rho_2)))}_\infty
         \le  \phi_2 n^{-1/3}$.
     \item[(d)] $\abs{f(\xvec)},\abs{g(\xvec)} \le n^{c_3} 
                    e^{c_2\xvec\trans\! A\xvec/n}$ for $\xvec\in\Reals^n$.
    \end{itemize}
  Let $\X$ be a random variable with the normal density
    $\pi^{-n/2} \abs{A}^{1/2} e^{-\xvec\trans\!A\xvec}$.
     Then, provided $\V f(\X)$ and $\V g(\X)$ are finite
     and $h$ is bounded in~$\varOmega$,
     \[
        \int_\varOmega e^{-\xvec\trans\!A\xvec + f(\xvec)+h(\xvec)}\,d\xvec
        = (1+K) \pi^{n/2}\abs{A}^{-1/2} e^{\E f(\X)+\frac12\V f(\X)},
     \]     
   where, for some constant $C$ depending only on $c_1,c_2,c_3,\eps$,
   \begin{align*}
      \abs{K} &\le C e^{\frac12\Var\Im f(\X)}\,\Bigl( e^{\phi_1^3+e^{-\rho_1^2/2}}-1
        \\
      &{\qquad}+ \(2e^{\phi_2^3+e^{-\rho_1^2/2}}-2
        + \sup_{\xvec\in\varOmega}\,\abs{e^{h(\xvec)}-1} \)\,
           e^{\E(g(\X)-\Re f(\X))+\frac12(\Var g(\X)-\Var\Re f(\X))}\Bigr).
   \end{align*}
   In particular, if $n\ge (1+2c_2)^2$ and
   $\rho_1^2 \ge 15 + 4c_2 + (3+8c_3)\log n$, 
    we can take~$C=1$.
\end{thm}
\begin{proof}
 We will divide the integral in the same fashion as~\eqref{3parts}, and
 will use estimate~\eqref{newgen2} again. We also need a similar
 estimate using Theorem~\ref{efmartin}(b).
 Lemma~\ref{errorterms} will be used to combine error terms.
 
  Suppose $\rho\ge c_1(\log n)^{1/2+\eps}$ and let
 $F:U_n(\rho)\to\Complexes$ be measurable and such that 
 $\abs{F(\xvec)}\le n^{c_3}e^{c_2\xvec\trans\!\xvec}$ for
 $\xvec\in\Reals^n$, and for $\xvec\in T(U_n(\rho))$,
 $\norm{\alpha(F,U_n(\rho))}_\infty\le\phi n^{-1/3}\le\tfrac23$ and
 $\norm{\varDelta(F,U_n(\rho))}_\infty\le\phi n^{-1/3}$.
  By Theorem~\ref{efmartin}(b),
  \begin{align*}
    \int_{U_n(\rho)} e^{-\yvec\trans\!\yvec+F(\yvec)}\,d\yvec &= 
      \(1 + K'e^{\frac12\Var(\Im F(\Y)\st \Y\in U_n(\rho))}\) \\[-1ex]
       &{\qquad}\times 
        e^{\E(F(\Y)\st \Y\in U_n(\rho))+\frac12\V(F(\Y)\st \Y\in U_n(\rho))} 
       \int_{U_n(\rho)} e^{-\yvec\trans\!\yvec}\,d\yvec,
  \end{align*}
  where
  $\abs{K'} \le e^{\phi^3}-1$
  and $\Y$ has the normal density $\pi^{-n/2}e^{-\yvec\trans\!\yvec}$.
  Similarly to the proof of Theorem~\ref{gauss2pt}, we can apply 
  Lemma~\ref{truncated} to conclude that there is
  a constant $n_0=n_0(c_1,c_2,c_3,\eps)$ such that for $n\ge n_0$,
  \begin{equation}\label{newgen3}
     \int_{U_n(\rho)} e^{-\yvec\trans\!\yvec+F(\yvec)}\,d\yvec = 
           (1+K''e^{\frac12\Var\Im F(\Y)})\, \pi^{n/2} e^{\E F(\Y)+\frac12\V F(\Y)} ,
  \end{equation}
  where $\abs{K''} \le e^{\phi^3+e^{-\rho^2/2}}-1$.

 Now consider the expansion given by~\eqref{3parts}.
 By condition~(b) and Lemma~\ref{smooth}, we have
 $\norm{\alphavec(f(T\yvec),U_n(\rho_1))}_\infty\le n^{-1/3}\rho_1\le\tfrac 23$,
 and
 $\norm{\varDelta(f(T\yvec),U_n(\rho_1))}_\infty\le n^{-1/3}\rho_1$.
 Consequently, by~\eqref{newgen3}, we have for $n\ge n_0$ that
 \begin{equation}\label{newgen4}
     \int_{U_n(\rho_1)} e^{-\yvec\trans\!\yvec+f(T\yvec)}\,d\yvec = 
           (1+K_1e^{\frac12\Var\Im f(\Y)})\, \pi^{n/2} e^{\E f(\Y)+\frac12\V f(\Y)} ,
  \end{equation}
  where $\abs{K_1} \le e^{\phi_1^3+e^{-\rho_1^2/2}}-1$.
  
  For the second part of~\eqref{3parts}, we need separate consideration
  of the two cases of condition~(c).  In case~(ii) we can apply~\eqref{newgen2}
  to $g(T\yvec)$ to obtain
  \begin{equation}\label{case1}
    \biggl(\, \int_{U_n(\rho_2)}-\int_{U_n(\rho_1)}\,\biggr)
                e^{g(T\yvec)-\yvec\trans\!\yvec}\,d\yvec
   = K'_2\pi^{n/2} e^{\E g(\X)},
  \end{equation}
  where $\abs{K'_2}\le 2(e^{\phi_2^3+e^{-\rho_1^2/2}}-1)$ for $n\ge n_0$.
  In case~(ii) we can assume $\phi_2n^{-1/3}\le\tfrac23$ or else case~(i)
  applies.  Then~\eqref{newgen3} gives for $n\ge n_0$ that
    \begin{equation}\label{case2}
      \biggl(\, \int_{U_n(\rho_2)}-\int_{U_n(\rho_1)}\,\biggr)
                  e^{g(T\yvec)-\yvec\trans\!\yvec}\,d\yvec
     = K''_2 \pi^{n/2} e^{\E g(\X)+\frac12\V g(\X)},
    \end{equation}
  where $\abs{K''_2}\le 2(e^{\phi_2^3+e^{-\rho_1^2/2}}-1)$.
  Since $e^{\frac12\V g(\X)}\ge 1$ ($g$ being real), we can
  write both~\eqref{case1} and~\eqref{case2} as
  $K_2\pi^{n/2} e^{\E g(\X)+\frac12\V g(\X)}$, where
  $\abs{K_2}\le\min(\abs{K'_2},\abs{K''_2})
          \le 2(e^{\phi_2^3+e^{-\rho_1^2/2}}-1)$.
          
  The third part of~\eqref{3parts} is bounded in modulus 
  by $\sup_{\xvec\in\varOmega}\,\abs{e^{h(\xvec)}-1}\, \pi^{n/2} e^{\E g(\X)}$
  just as in Theorem~\ref{gauss2pt}.
  Adding the three parts, and noting that $C$ can be increased to cover
  the finite number of cases when $n<n_0$, the theorem follows.
  
  The final claim is proved essentially as in
  the previous theorem.
\end{proof}

\begin{remark}\label{boxremark}
Note that the assumption
$U_n(\rho_1)\subseteq T^{-1}(\varOmega)\subseteq U_n(\rho_2)$
of Theorems~\ref{gauss2pt} and~\ref{gauss4pt}
is implied by $U_n(\rho_1\norm{T}_\infty)\subseteq
 \varOmega\subseteq U_n(\rho_2\norm{T^{-1}}_\infty^{-1})$, so the 
 latter condition could be used instead of the former.
\end{remark}

\subsection{Truncated gaussian measures of less than full rank}\label{S:subspace}

Many enumeration problems have generating functions with
symmetries that lead to singular quadratic forms.  As an example,
which we will work in more detail in Section~\ref{S:regtour},
regular tournaments are counted by the constant term of
$\prod_{1\le j<k\le n} (x_j/x_k+x_k/x_j)$, which is invariant under 
multiplication of each variable by the same constant~\cite{Mtourn}. 
Expanding at the saddle-point gives the quadratic form 
$\sum_{1\le j<k\le n}(\theta_j-\theta_k)^2$, which
is invariant in the direction $(1,1,\ldots,1)$.
By conditioning on the value of one variable, or the sum of the
variables, we can restrict the integral to a subspace of codimension~1.
In other problems the codimension can be higher.  Here we provide
a general technique that expands such integrals to full dimension,
so that the techniques of the previous subsection can be applied.

If $T:\mathbb{R}^n\rightarrow \mathbb{R}^n$ is a linear operator, let 
$\ker T = \{\x\in \mathbb{R}^n \st T\x = 0\}$. If $L$ is a linear subspace of $\mathbb{R}^n$, let  $L^\perp$ be the orthogonal complement.

\begin{lemma}\label{LemmaQW}
  Let $Q,W:\mathbb{R}^n\rightarrow \mathbb{R}^n$ be linear operators 
  such that $\ker Q \cap \ker W = \{\boldsymbol{0}\}$ and ${\rm span}(\ker Q, \ker W) = \Reals^n$. 
  Let $n_\perp$  denote the dimension of\/ $\ker Q$.
  Suppose $\varOmega \subseteq \Reals^n$ and
  $F:\varOmega\cap Q(\Reals^n) \to\Complexes$.
  For any $\rho>0$, define
  $$
   \varOmega_{\rho} = \bigl\lbrace \xvec\in\Reals^n \st
     Q\xvec\in \varOmega \text{~and~}
     W\xvec\in U_n(\rho) \bigr\rbrace.
  $$
   Then, if the integrals exist,
   \[
    \int_{\varOmega \cap Q(\Reals^n)} F(\yvec)\,d\yvec
    = (1 - K)^{-1}\,\pi^{-\nperp/2} \,\Abs{Q\trans Q + W\trans W}^{1/2}
     \int_{\varOmega_{\rho}} F(Q\xvec)\, e^{-\xvec\trans\! W\trans W \xvec}
        \,d\xvec,
   \]
   where 
   $$0\le K < \min(1,n e^{-\rho^2/\kappa^2}), \ \  
   \kappa = \sup_{W\x \neq 0}  \frac{\norm{W\xvec}_\infty} {\norm{W\xvec}_2} \leq 1.$$ 
      Moreover, if  $U_n(\rho_1) \subseteq \varOmega \subseteq  U_n(\rho_2)$ for some $\rho_2 \geq \rho_1 >0$ then
  \[
      U_n\biggl(\min\Bigl(\frac{\rho_1}{\norm{Q}_\infty},
          \frac{\rho}{\norm{W}_\infty} \Bigr)\biggr)
        \subseteq \varOmega_{\rho} \subseteq
      U_n\biggl( \norm{P}_\infty\, \rho_2 +  \norm{R}_\infty\, \rho \biggr)
  \]	
   for any linear operators $P,R : \Reals^n \rightarrow \Reals^n$ such that $PQ + RW$ is equal to the identity operator on $\Reals^n$.

\end{lemma}
\begin{proof}
The bounds on $\varOmega_\rho$ follow directly from the definition of $\varOmega_\rho$: 
for the lower bound, we use  $\norm{Q \x}_\infty \leq 
\norm{Q}_\infty \norm{\x}_\infty \leq \rho_1$ and
  $\norm{W \x}_\infty \leq \norm{W}_\infty \norm{\x}_\infty \leq \rho$;
for the upper bound, apply 
$\norm{\x}_\infty \leq \norm{P}_\infty \norm{Q \x}_\infty  + \norm{R}_\infty \norm{W \x}_\infty$.
  
  Due to the assumptions on  $\ker Q$ and $\ker W$, we can find some invertible linear operator
$T:\Reals^n \rightarrow \Reals^n$ such that $T(\ker Q) = (T (\ker W))^\perp$. 
Substituting $\xvec = T^{-1}\hat \xvec$, we get that
$$
     \int_{\varOmega_{\rho}} F(Q\xvec)\, e^{-\xvec\trans\! W\trans W\xvec}
        \,d\xvec = \abs{T}^{-1}
   \int_{T(\varOmega_{\rho})} F(\hat{Q}\hat \xvec)\, 
               e^{-\hat \xvec\trans\! \hat{W}\trans \hat{W}\hat \xvec}
        \,d\hat\xvec,
$$ 
where $\hat{Q}=  QT^{-1}$, $\hat{W}=  WT^{-1}$. 
Note that $\ker \hat{Q}=  T(\ker Q)$ 
and $\ker \hat{W}=  T(\ker W)$.
   Consider an orthonormal basis
   consisting of $n-\nperp$ vectors that span $\ker \hat W $ and $\nperp$
   vectors that span~$\ker \hat Q$.
   Thus $\xvec'\in \Reals^n$ is represented as
    $(\x_Q,\x_W)\in \ker \hat W \oplus \ker \hat  Q$.
   The quadratic form with matrix $\hat Q\trans \hat Q + \hat W\trans \hat W$ acts separately on 
   the orthogonal subspaces $\ker \hat W$ and $\ker \hat Q$, therefore 
   $\abs{\hat Q\trans \hat Q + \hat W\trans \hat W}^{1/2}
      = \abs{T}^{-1}\,\abs{Q\trans Q + W\trans W}^{1/2}$ 
   is equal to the product of the Jacobian determinants of the linear maps 
   corresponding to the restrictions of $\hat Q$ to $\ker \hat W$ and
   $\hat W$ to $\ker \hat Q$.
   Then we have
   \begin{align*}
    \int_{T(\varOmega_{\rho})} F(\hat{Q}\hat\xvec)\, e^{-\hat \xvec\trans\! \hat{W}\trans \hat{W}\hat\xvec}
        \,d\hat\xvec = 
      \int_{  \ker \hat W \cap T(\varOmega_\rho)} F(\hat Q \x_Q )\,d \x_Q \times
        \int_{ \ker \hat Q \cap T(\varOmega_\rho) }
           e^{-\x_W\trans \hat W\trans \hat W \x_W}\,d\x_W \\
           =  \abs{Q\trans Q  + W\trans W}^{-1/2} \abs{T}
           \int_{\varOmega\cap Q(\Reals^n)} F(\yvec)\,d \yvec \times 
        \int_{ U_n(\rho) \cap W(\Reals^n)}
           e^{-\zvec\trans \zvec }\,d\zvec, 
   \end{align*}
   where the last integral would be $\pi^{\nperp/2}$ except for 
   the restricted domain.
   Thus the claim follows with
   $K=\Prob(\X\notin U_n(\rho))$,
   where $\X$ is a normal variable on $W(\Reals^n)$ with
   density $\pi^{-\nperp/2} e^{-\zvec\trans\zvec}$.
   The cube $U_n(\rho)$ intersects $W(\Reals^n)$ in a convex
   polytope whose facets are intersections of $W(\Reals^n)$
   with the facets of $U_n(\rho)$.  By
   the definition of $\kappa$ the perpendicular
   distance from the origin to a facet of
   $U_n(\rho)\cap W(\Reals^n)$ 
    is at least equal to $\inf_{\text{facet of } U_n(\rho)\cap W(\Reals^n) } \norm{\zvec}_\infty/\kappa  =  \rho/\kappa$.
   Since there are at most $2n$ such facets, we have that
   $K\le 2n\pi^{-1/2}\int_{\rho/\kappa}^\infty e^{-x^2}\,dx
   \le ne^{-\rho^2/\kappa^2}$.
\end{proof}

 For an unbounded region, we get the following corollary.  
\begin{cor}\label{C:expectations}
 Let $\varOmega = \Reals ^n$ and assumptions of Lemma \ref{LemmaQW} hold.
 For any linear subspace $L\subseteq \Reals^n$ such that $L\cap \ker Q = \{0\}$ and ${\rm span} (L, \ker Q) = \Reals^n$,
  define a random variable~$\Y_L$ taking values in $L$ 
 with density proportional to $e^{-\yvec\trans\! Q\trans Q \yvec}$. Then
     $\E F(Q\Y_L)$ does not depend on the choice of $L$ 
     and is  equal to $\E F(Q\X_W)$, where   $\X_W$ is the random variable taking values in $\Reals^n$  
     with density proportional to $e^{-\x\trans\! (Q\trans Q + W\trans W) \x\trans}$.
\end{cor}

\begin{proof}
Observe that 
\[
	\E F(Q\Y_L) =   \frac{\int_{Q(\Reals^n)} F(\zvec) e^{-\zvec\trans\zvec}\,d\zvec}
	{\int_{Q(\Reals^n)} e^{-\zvec\trans\zvec}\,d\zvec}.
\]
To complete the proof, we use Lemma \ref{LemmaQW} with $\rho \rightarrow \infty$, which implies $K\rightarrow 0$. 
\end{proof}

\subsection{Example: regular tournaments}\label{S:regtour}

The enumeration of regular tournaments makes a good example to demonstrate how Lemma~\ref{LemmaQW} can be used to reduce
an integral to a form for which Theorem~\ref{gauss2pt} applies.
We recall that a regular tournament is a complete digraph in which the in-degree is equal to the out-degree at each vertex.
Let $RT(n)$ be the number of labelled regular tournaments with $n$ vertices.
It is clear that $RT(n) = 0$ if $n$ is even.
The following formula was given for the first time in~\cite{Mtourn}:
\begin{thm}
	For odd $n\rightarrow \infty$ 
	\begin{equation}\label{RegTour}
		RT(n) = \(1+ O(n^{-1/2+\eps})\) \biggl(\frac{2^{n+1}}{\pi n}\biggr)^{(n-1)/2}
		n^{1/2}e^{-1/2} 
	\end{equation}
	for any $\eps>0$.
\end{thm}
\begin{proof}
	 Observe that $RT(n)$ is equal to the constant term of the generating function  $\prod_{1\le j<k\le n} (x_j/x_k+x_k/x_j)$. Using contours $x_j = e^{i\theta_j}$, we get by Cauchy's theorem that
	 \[
	 	RT(n) = \frac{2^{n(n-1)/2}}{(2\pi)^n}\,\mbox{Int}, \ \ \ 
	 	\mbox{Int} = \int_{U_n(\pi)} \prod_{1\leq j<k\leq n} \cos(\theta_j-\theta_k) \,d\boldsymbol{\theta}.
	 \]
	 The next step, which we omit here and refer to \cite[Sect.~3]{Mtourn}, is to show that for odd $n\rightarrow \infty$ 
	 \[	
	 	\mbox{Int} = \(1+ O(e^{-cn^{2\eps}})\) 2^n\pi \int_{U_{n-1}(n^{-1/2 + \eps})}
	 	 \prod_{1\leq j<k\leq n} \cos(\theta_j-\theta_k) \,d\boldsymbol{\theta}'
	 \]
	 for some $c>0$, where the integration is with respect to 
	 $\boldsymbol{\theta}' = (\theta_1,\ldots, \theta_{n-1})$ with $\theta_n =0$.
	 
	 Now let $\boldsymbol{1} = (1,\ldots,1)\in \Reals^n$ and 
	 	\begin{align*}
	 	 \varOmega = U_{n}(n^{-1/2 + \eps}), \ \ \ F(\xvec) = \prod_{1\leq j<k\leq n} \cos(x_j-x_k), \\
	 Q\xvec = \xvec - x_n \boldsymbol{1}, \ \   \
	 W\xvec =  \tfrac{1}{\sqrt{2n}}(x_1 + \cdots +x_n)\boldsymbol{1},\\
	 P\xvec =  \xvec - \tfrac{1}{n}(x_1 + \cdots +x_n)\boldsymbol{1}, \ \ \
	 R\xvec = \sqrt{\tfrac 2n}\, \xvec. 
	 \end{align*}
	 We observe $F(Q\xvec) = F(\xvec)$ and apply Lemma \ref{LemmaQW} with 
	 $\rho = \tfrac{1}{\sqrt{2}}n^{\eps}$ to get
	 \begin{align*}
	 	\int_{U_{n-1}(n^{-1/2 + \eps})} &\prod_{1\leq j<k\leq n} \cos(\theta_j-\theta_k) \,d\boldsymbol{\theta}' = \int_{\varOmega \cap Q(\Reals^n)} F(\yvec) d \yvec\\
	 	&= \(1+ O(e^{-c'n^{2\eps}})\) \pi^{-1/2} 2^{-1/2} n \int_{\varOmega_\rho} F(\xvec) e^{- \frac12  (x_1+\cdots +x_n)^2} d \xvec
	 \end{align*}
	 for some $c'>0$. We obtain also that $U_n(\tfrac12 n^{-1/2 + \eps} )\subseteq \varOmega_\rho \subseteq U_n(3 n^{-1/2 + \eps} )$.
	 By Taylor's theorem, we can expand 
	 \[
	 	F(\xvec) e^{- \frac12  (x_1+\cdots +x_n)^2} = \exp \Bigl( -\dfrac{n}{2} \xvec\trans \xvec - \dfrac{1}{12} \sum_{1\leq j<k\leq n}(x_j - x_k)^4  + O(n^{-1+6\eps}) \Bigr).
	 \]
	 Define $\X$ to be the gaussian random variable with density 
	 $(2\pi)^{-n/2}n^{n/2}e^{-\frac n2\xvec\trans\xvec}$ and
	 let $f(\xvec) = -\tfrac{1}{12} \sum_{1\leq j<k\leq n}(x_j - x_k)^4$.
     Then $\E f(\X) = -\frac{(n-1)}{2n}$ and
     $\partial f/\partial x_j=O(n^{-3/2+4\eps})$ for
     $\xvec\in\varOmega_\rho$ and $1\le j\le n$.
     Now apply Theorem \ref{gauss2pt} with $A= \tfrac{n}{2}I$, $T = \sqrt{\frac 2n}\, I$, $\rho_1,\rho_2 = O(n^{\eps})$, $\phi_1,\phi_2 = O(n^{-1/2 + 4\eps})$ and $g(\xvec)=f(\xvec)$
     to find that
	 \[	
	 			\int_{\varOmega_\rho} F(\xvec) e^{- \frac12  (x_1+\cdots +x_n)^2} d \xvec
	 			= \(1+ O(n^{-1 + 8\eps})\) 2^{n/2}  \pi^{n/2} n^{-n/2} e^{-1/2}.
	 \] 
	Formula \eqref{RegTour} follows.
\end{proof}

Although we used an old theorem here for illustrative purposes, it is
worth nothing that the same method can be used to enumerate tournaments 
according to score sequence over a very wide range of scores, well
beyond that achieved in~\cite{GMW}.
The details will appear separately.

An example of how Lemma~\ref{LemmaQW} can be applied in conjunction
with Theorem~\ref{gauss4pt}
is the enumeration of bipartite graphs, which we will cover in
Section~\ref{S:examples}.

\subsection{The case of weakly dependent components}\label{S:weakly}

In order to apply Theorems \ref{gauss2pt} and \ref{gauss4pt} to particular examples, 
we need to know that there exists some linear transformation $T$ 
such that $T\trans\! A T = I$ and which satisfies good bounds on
$\norm{T}_1,\norm{T}_\infty$ (and $\norm{T^{-1}}_\infty$).
In this subsection we give a general recipe for finding $T$ in the case when 
diagonal elements of $A$ are of the same order while off-diagonal elements are relatively small. 
Equivalently, the components of the corresponding gaussian random variable are weakly dependent. 

As was mentioned in Sections \ref{S:subspace} and \ref{S:regtour}
sometimes we have that $A$ is a positive-semidefinite matrix with
non-trivial kernel and the region of integration lies in some linear subspace of $\Reals^n$. 
Then, using Lemma~\ref{LemmaQW}, we can reduce it to the integration 
over a region of full dimension with a modified quadratic form $A+W\trans W$ which is non-singular. 
For such purposes we also need analogous estimates for a linear
transform $T$ satisfying $T\trans(A+W\trans W)T = I$.
There is a large flexibility of choosing $W$ in general. One strategy is to make
$A+W\trans W$ close to some diagonal matrix and proceed as in the case of full dimension. 
Alternatively, when $A$ is close to some diagonal matrix $D$
but entries of $W\trans W$ are always big in comparison
with entries of $A - D$ it turned out to be better to
choose $W$ in one particular way as described below. 

If $D$ is a positive-semidefinite matrix, denote by $D^{1/2}$
the positive-semidefinite square root and, in the case of nonsingularity,
by $D^{-1/2}$ the positive-definite inverse square root. Let 
$$A_{D} = A + D^{1/2}P_D D^{1/2},$$ 
where $P_D$ is the linear operator that projects orthogonally onto~$D^{1/2}(\ker A)$. 
Assuming that $D$ is not singular, note that $A_{D}\xvec=A\xpar +  D\xperp$, where 
$$
 \xperp = D^{-1/2} P_D D^{1/2} \xvec \in \ker A,
  \ \ \ \ \xpar = \xvec - \xperp = D^{-1/2} (I-P_D) D^{1/2} \xvec.
$$
In the case of $\ker A = \{\boldsymbol{0}\}$ we have $A_{D} = A$ and $\xpar=\xvec$.

\begin{lemma}\label{diagonal}
 Let $D$ be an $n\times n$ real diagonal matrix with
 $\dmin=\min_j d_{jj}>0$ and $\dmax=\max_j d_{jj}$.
 Recall the norm $\maxnorm{\cdot}$ defined in Section~\ref{S:smooth}.
 Let $A$ be a real symmetric positive-semidefinite $n\times n$
 matrix with
  \[
      \maxnorm{A-D} \le \frac{r\dmin}{n} \ \
      \text{~~and~~} \ \
      \xvec\trans\! A\xvec \ge \gamma \xpar\trans\! D\xpar = \gamma\xvec\trans\! D^{1/2} (I - P_D) D^{1/2}\xvec
  \]
  for some $1\geq \gamma >0$, $r>0$ and all $\xvec\in\Reals^n$.
  Let $\nperp$ denote the dimension of $\ker A$.
 Then the following are true.
 \begin{itemize}\itemsep=0pt
  \item[(a)] $\displaystyle \norm{A_{D} - A}_\infty \leq  r \nperp  {\dmax^{1/2}}{\dmin^{1/2}},
      \ \ 
     \norm{A_{ D} - A}_{\max} \leq  \frac{r^2 \nperp \dmin}{n}     \text{~~and~~} \nperp \le r^2.$
     
   \item[(b)] $A_D$ is symmetric and positive-definite.  Moreover,
    \[
     \norm{ A_D^{-1}}_\infty \le \frac{r+\gamma}{\gamma \dmin} 
     \ \ 
     \text{~~and~~} 
     \ \ 
     \norm{ A_D^{-1} - D^{-1}}_{\max} \leq \frac{(r+\gamma) r  }{\gamma n \dmin} (1+r \nperp) .
    \]
     
  \item[(d)] There exists a matrix $T$ such that $T\trans\! A_D T=I$ and
   \[ \norm{T}_1, \norm{T}_\infty 
           \le \frac{r+\gamma^{1/2}}{\gamma^{1/2}\dmin^{1/2}}, 
          \ \ \ \ 
         \norm{T^{-1}}_1, \norm{T^{-1}}_\infty
           \le \biggl( \frac{(r+1)(r+\gamma^{1/2})}{\gamma^{1/2}} + r\nperp\biggr) \dmax^{1/2}. \]
 Furthermore,
   	\begin{align*}
   	 	\norm{T - D^{-1/2}}_{\max}  &\leq \Bigl( \dfrac{r^2+r}{2} + \dfrac{r^2}{\gamma^{1/2}}\Bigr)\dmin^{-1/2} n^{-1},
   	 \\
   		\norm{T^{-1} - D^{1/2}}_{\max}  &\leq \Bigl(\dfrac{3r}2  + r^2 (\dfrac12+ \dfrac{2}{\gamma^{1/2}} + n_\perp) + r^3 \dfrac{n_\perp}{\gamma^{1/2}}\Bigr)
   		 \dmax^{1/2} n^{-1}.
   	\end{align*}
 \end{itemize}
 
\end{lemma}
\begin{proof}
  Let $\yvec$ be a unit vector of $D^{1/2}(\ker A)$. Then  $D^{1/2} \yvec  = (D - A) D^{-1/2} \yvec$, so by assumption 
  and using  $\norm{\yvec}_1 \leq {n}^{1/2} \norm{\yvec}_2 = {n}^{1/2}$, we find that  
  \[
  \norm{D^{1/2} \yvec}_\infty \leq \tfrac{r \dmin }{n} \norm{D^{-1/2}\yvec}_1  \leq \tfrac{r \dmin^{1/2}}{n^{1/2}}, 
    \ \ \ \norm{D^{1/2} \yvec}_1 \leq \dmax^{1/2}\norm{\yvec}_1 \leq  n^{1/2}\dmax^{1/2}.
  \] 
  Consequently we get $\norm{D^{1/2} \yvec \yvec\trans D^{1/2}}_\infty \leq r\dmin^{1/2}\dmax^{1/2}$ 
  and  $\norm{D^{1/2} \yvec \yvec \trans D^{1/2}}_{\max} \leq \frac{r^2\dmin}{n}$. 
  If $\{\yvec_1,\ldots,\yvec_{\nperp}\}$ 
  is a full set of orthonormal vectors of $D^{1/2}(\ker A)$
  then $P_D = \sum_{j=1}^{\nperp} \yvec_j \yvec_j\trans$ and $A_D - A = \sum_{j=1}^{\nperp} D^{1/2}\yvec_j \yvec_j\trans D^{1/2}$ 
  which implies the first two estimates of part (a). The last estimate of part (a) follows  
   from the observation that the trace of $(A-D)^2$
  is at most $r^2 d_{\min}^2$ and at least $\nperp d_{\min}^2$.
   
  For any $\xvec \in \Reals^n$ and $t\in \Reals$ we have 
  \begin{align*}
  \norm{(t^2 D + A_{D}) \xvec}_\infty
   =\norm{(t^2+1)D\xvec+(A-D)\xpar}_\infty
  &\ge (t^2+1)\dmin\norm{\xvec}_\infty -\tfrac {r\dmin}{n} \norm{\xpar}_1
  \\ &\ge  (t^2+1)\dmin\norm{\xvec}_\infty -\tfrac{r\dmin}{ n^{1/2}}\norm{\xpar}_2.
  \end{align*}
  Note that the eigenvalues of 
  $D^{-1/2}\hat A D^{-1/2} = D^{-1/2} A D^{-1/2} + P_D$ 
  are positive eigenvalues of $D^{-1/2} A D^{-1/2}$ that are at least $\gamma$, plus 
  $\nperp$ extra eigenvalues equal to~$1$.
  Putting $\yvec = D^{1/2}\xvec$, we get that
  \begin{align*}
  \norm{(t^2 D+ A_{D}) \xvec}_\infty 
    \ge\tfrac{\dmin^{1/2}}{n^{1/2}}\norm{ (t^2 I +D^{-1/2} A _DD^{-1/2} )\yvec}_2
   \ge \tfrac{(t^2+ \gamma)\dmin^{1/2}}{n^{1/2}}\norm{\yvec}_2
 \\
   \ge \tfrac{(t^2+\gamma)\dmin^{1/2}}{ n^{1/2}}\norm{P_D \yvec}_2 
   \ge \tfrac{(t^2+\gamma)\dmin}{ n^{1/2}}\norm{\xpar}_2.
   \end{align*}
  Adding $t^2+ \gamma$ times the first inequality to $r$ times the
  second, we find that 
  \begin{equation}\label{bextended}
  (t^2+\gamma+ r)\norm{(t^2 D + A_D)\xvec}_\infty
  \ge (t^2+\gamma)(t^2+1) \dmin\norm{\xvec}_\infty,
  \end{equation}
   and
 \begin{equation}\label{bextended_max}
  \begin{aligned}
 \norm{(t^2 D +A_D)^{-1} - \tfrac{1}{t^2+1} D^{-1}}_{\max} &= \tfrac{1}{t^2+1}
   \norm{(t^2 D +A_D)^{-1} (A_D-D) D^{-1}}_{\max}  \\
   &\leq   \tfrac{1}{t^2+1} \dmin^{-1} \norm{(t^2 D +A_D)^{-1}}_\infty\, \norm{A_D - D}_{\max}.
 \end{aligned}
 \end{equation}
 Estimates \eqref{bextended} and \eqref{bextended_max} for $t=0$ imply  part~(b).
 
  Let $B = D^{-1/2} A D^{-1/2}$.  
  Note that $B$  satisfies conditions of Lemma \ref{diagonal} with $I$ playing role of $D$ and $B+P_{\ker B}$ playing role of $A_D$.
  Namely, 
  \[\norm{B - I}_{\max} \leq \dfrac{1}{\dmin} \norm{A - D}_{\max} \leq r/n\] 
  and $\xvec \trans\! B \xvec \geq \gamma \xvec \trans (I- P_{\ker B}) \xvec$, where $P_{\ker B}$ is the orthogonal projector onto $\ker B$. 
     From~\cite[p.~133]{Higham} we have that
  \[
     (B + P_{\ker B})^{-1/2} = \frac{2}{\pi} \int_0^\infty (t^2I+B + P_{\ker B})^{-1}\,dt.
  \]
    Using the bounds \eqref{bextended} and \eqref{bextended_max} with  $B$ playing role of $A$, we find that 
  \[
       \norm{(t^2 I +  B + P_{\ker B})^{-1}}_\infty \le \dfrac{t^2+\gamma+r}
             {(t^2+\gamma)(t^2+1) }
  \]
  and
  \[
  	  \norm{(t^2 I +  B + P_{\ker B})^{-1} - \tfrac{1}{t^2+1} I}_{\max} \le \dfrac{(t^2+\gamma+r)r}
             {(t^2+\gamma)(t^2+1)^2 n}.
  \]
  Performing the integral gives
  \begin{equation*}
      \norm{(B + P_{\ker B})^{-1/2}}_\infty
         \le \dfrac{r+{\gamma}^{1/2}+\gamma }{{\gamma}^{1/2}+\gamma}
         \le 1+\dfrac{r}{{\gamma}^{1/2}}
  \end{equation*}
  and
  \[
  		  \norm{(B + P_{\ker B})^{-1/2}-I}_{\max}
         \le \dfrac{(2r+r{\gamma}^{1/2} +\gamma^{1/2}+ 2\gamma+ \gamma^{3/2})r }{2(\gamma^{1/2}+ 2\gamma+ \gamma^{3/2})n}
         \le \Bigl(\dfrac{r^2+r}{2}+ \dfrac{r^2}{\gamma^{1/2}}\Bigr) n^{-1}.
  \]
  
   From the eigensystems we see that $(B + P_{\ker B})^k$ acts the same
  as $B^k$ on vectors in $\ker B^\perp$ and preserves vectors in $\ker B$.
  Consequently, 
  $$(B + P_{\ker B})^{1/2} = B (B + P_{\ker B})^{-1/2}+P_{\ker B}.$$
 Claim (a) with  $B$ playing role of $A$ gives   $\norm{P_{\ker B}}_\infty \leq r \nperp$, 
  $\maxnorm{P_{\ker B}} \leq \frac{r^2 n_{\perp}}{n}$.
   Using also $\norm{B}_\infty \leq n\maxnorm{B-I} + \norm{I}_\infty \leq r+1$, we get that
 \[
  	\norm{(B + P_{\ker B})^{1/2}}_\infty  \leq \dfrac{(r+1)(r+\gamma^{1/2})}{\gamma^{1/2}} + r\nperp
 \]
  	and
  	\begin{align*}
  			\maxnorm{(B + P_{\ker B})^{1/2} &-  I}  \\
			&  	\leq \maxnorm{(B + P_{\ker B})^{-1/2} -  I} +
  		\maxnorm{(B + P_{\ker B})^{1/2} -  (B + P_{\ker B})^{-1/2}} \\
  		&\leq  \maxnorm{(B + P_{\ker B})^{-1/2} -  I}
		 + \norm{(B + P_{\ker B})^{-1/2}}_\infty\; \maxnorm{B-I + P_{\ker B}} \\
  		&\leq \Bigl(\dfrac{r^2+r}{2}+ \dfrac{r^2}{\gamma^{1/2}}\Bigr) n^{-1} 
		+ \Bigl(1+ \dfrac{r}{{\gamma}^{1/2}}\Bigr) 
  		(r+ r^2 n_\perp)n^{-1}\\
  		&= \Bigl(\dfrac{3r}{2}  + r^2 \(\dfrac12+ \dfrac{2}{\gamma^{1/2}} 
		+ n_\perp\) + r^3 \dfrac{n_\perp}{\gamma^{1/2}}\Bigr)n^{-1}.
  	\end{align*}
 
  In order to prove claim~(d) we can take $T=D^{-1/2}(B + P_{\ker B})^{-1/2}$, which indeed
  satisfies $T\trans\! A_D T=I$. 
  For $p\in\{1,\infty\}$,
  \begin{align*}
  \norm{T}_p&=\norm{D^{-1/2}(B + P_{\ker B})^{-1/2}}_p
  \le \dmin^{-1/2} \norm{(B + P_{\ker B})^{-1/2}}_p, \\
    \norm{T - D^{-1/2}}_{\max} 
    &=  \norm{D^{-1/2}((B + P_{\ker B})^{-1/2}-I)}_{\max} 
    \leq \dmin^{-1/2} \norm{(B + P_{\ker B})^{-1/2}-I}_{\max}, 
  \end{align*}
  and similarly 
  \begin{align*}
  \norm{T^{-1}}_p \le\dmax^{1/2}\norm{(B + P_{\ker B})^{1/2}}_p, \ \ \ \ 
  \norm{T - D^{1/2}}_{\max} \leq \dmax^{1/2} \norm{(B + P_{\ker B})^{1/2}-I}_{\max}.
  \end{align*}
  Recalling that $(B + P_{\ker B})^{-1/2}$ and $(B + P_{\ker B})^{1/2}$ are symmetric, claim~(d)
  follows.
\end{proof}

\begin{remark} 
Actually, Lemma \ref{diagonal} is  valid not only when  $D$ is a diagonal matrix but 
also when $D$ is any symmetric positive-definite $n \times n$ matrix,
with $\norm{D^{1/2}}_\infty^2$ playing the role of $\dmax$ and
$\norm{D^{-1/2}}_\infty^{-2}$ playing the role of $\dmin$. 
The proof of this generalization  is identical.
\end{remark}


\nicebreak
\section{Graphs with given degrees}\label{S:examples}

In this section we will demonstrate the use of our theory to obtain
new results on graphs with given degrees. 
We will generalize the problem as follows.

Let $H=(H^+,H^-)$ be a pair of fixed (simple) edge-disjoint
graphs on vertices $V=\{1,\ldots,n\}$.
We will not notationally distinguish graphs from their edge-sets.
Let $N_H(\dvec)$ be the number of graphs
on $V$ which have vertex degrees $\dvec=(d_1,\ldots,d_n)$,
include $H^+$ as a subgraph, and are edge-disjoint
from $H^-$.  The generating function for~$N_H$~is
\begin{align}
   F_{\dvec,H}(x_1,\ldots,x_n) &= 
   \sum_{d_1,\ldots,d_n} 
     N_H(\dvec) \, x_1^{d_1}\cdots x_n^{d_n} \notag\\
   &= \prod_{\{j,k\}\in H^+} \negthickspace x_jx_k \;
   \prod_{\{j,k\}\notin H^+\cup H^-} \negthickspace (1 + x_jx_k).
    \label{FH}
\end{align}
From this it follows that
\begin{equation}\label{NHint}
   N_H(\dvec) = 
   \frac{1}{(2\pi i)^n} \oint\!\cdots\!\oint\;
     \frac{F_{\dvec,H}(z_1,\ldots,z_n)}{z_1^{d_1+1}\cdots z_n^{d_n+1}}
     \, dz_1\cdots dz_n,
\end{equation}
where each contour circles the origin once anticlockwise.

The value of $N_{\Hempty}(\dvec)$ was estimated by
McKay and Wormald~\cite{MWreg} when $d_1,\ldots,d_n$ are
large (approximately a constant fraction of $n$) and not very far from equal.
McKay~\cite{ranx} later extended this to the case of nonempty $H$,
provided $H^+\cup H^-$ has at most $n^{1+\eps}$ edges and
maximum degree at most $n^{1/2+\eps}$.
Meanwhile, Barvinok and Hartigan~\cite{BarvHart1} extended
the case of $H=\Hempty$ to a much wider range of
degrees.

Our definition also includes the bipartite case.
Let $V_1=\{1,\ldots,n_1\}, V_2=\{n_1{+}1,\ldots,\allowbreak n_1{+}n_2\}$
be a partition of $V$ into two disjoint subsets.
Define $\tilde{E}=\binom{V_1}{2}\cup\binom{V_2}{2}$;
that is, the complement of the complete bipartite graph
with parts~$V_1$ and~$V_2$.
If $\tilde E\subseteq H^-$, then $N_H(\dvec)$ is a count
of bipartite graphs. 

Canfield and McKay~\cite{CM} estimated
$N_{\{\emptyset,\tilde E\}}(\dvec)$
in the semiregular case, which was later extended to more
irregular degree sequences by Barvinok and Hartigan~\cite{BarvHart1}.
The case where $H^+\cup H^-\ne\emptyset$
was treated by Greenhill and McKay~\cite{GMX} if 
$\dvec$ is not far from semiregular.

We will generalize all these results.
Changing variables in~\eqref{NHint} with $z_j=e^{\beta_j+i\theta_j}$,
and defining
\begin{equation}\label{lameqn}
  p_{jk} = \begin{cases}
     \; 0 & \text{if } j=k \text{~or~} \{j,k\}\in H^-; \\
     \;   1 & \text{if } \{j,k\}\in H^+; \\
     \displaystyle \frac{e^{\beta_j+\beta_k}}{1+e^{\beta_j+\beta_k}}
         & \text{otherwise,}
   \end{cases}
\end{equation}
and recalling the definition of $U_n(\rho)$ in Section~\ref{S:smooth}, we have
\begin{align}
  N_H(\dvec) &= C_{\dvec,H}  \int_{U_n(\pi)}
   G_{\dvec,H}(\thetavec)\,d\thetavec, \label{NHd} \\
\intertext{where} 
  C_{\dvec,H} &= \frac{\prod_{\{j,k\}\in H^+} e^{\beta_j+\beta_k}\;
                     \prod_{\{j,k\}\notin H^+\cup H^-}(1+e^{\beta_j+\beta_k})}
                     {(2\pi)^n\;e^{d_1\beta_1+\cdots+d_n\beta_n}}, \notag \\
  G_{\dvec,H}(\thetavec) &= \frac{\prod_{\{j,k\}\in H^+} e^{i(\theta_j+\theta_k)}
    \prod_{\{j,k\}\notin H^+\cup H^-}
        \(1+p_{jk}(e^{i(\theta_j+\theta_k)}-1)\) }
        {e^{i(d_1\theta_1+\cdots+d_n\theta_n)}}. \label{Gdef}
\end{align}

Equation~\eqref{NHd} is valid for any radii $\{e^{\beta_j}\}$, but
in order to estimate the integral we need its value to be
concentrated in a small region where the integrand is not
too oscillatory.  There are also symmetries to 
take into account.  The most obvious is that
$G_{\dvec,H}(\theta_1,\ldots,\theta_n)=
 G_{\dvec,H}(\theta_1+\pi,\ldots,\theta_n+\pi)$.  In the bipartite
 case we also have that
$G_{\dvec,H}(\theta_1,\ldots,\theta_n)=
 G_{\dvec,H}(\theta_1+t,\ldots,\theta_{n_1}+t,
   \theta_{n_1+1}-t,\ldots,\theta_n-t)$ for any~$t$.
Other symmetries can occur if the complement of 
$H^+\cup H^-$ is disconnected, but we will not consider
those cases here.

A good choice of radii is that which makes the contours
pass together through the saddle point on the positive real axis.
This gives the equations
\begin{equation}\label{saddlepoint}
    \sum_{k=1}^n p_{jk} = d_j, \quad(1\le j\le n),
\end{equation}
in which case we have
\[
  C_{\dvec,H} = (2\pi)^{-n} \prod_{1\le j<k\le n}
                          \, p_{jk}^{-p_{jk}}(1-p_{jk})^{-(1-p_{jk})},
\]
where $0^0=1$ as usual.
There is no comprehensive theory about when $\{\beta_j\}$ exist
to satisfy~\eqref{saddlepoint}, but much is known in the cases
$H=\Hempty$ 
and $H=(\emptyset,\tilde E)$, which will suffice for us here.

In the case $H=\Hempty$, a unique solution
for $\{\beta_j\}$ exists if $\dvec$ lies in the interior of the
polytope defined by the Erd\H os-Gallai
inequalities~\cite{BarvHart1,Csiszar,Rinaldo}.
The corresponding values $\{p_{jk}\}$ have an
important property: if we generate a random graph, where
for each $j,k$, there is an edge from vertex~$j$ to vertex~$k$
with probability $p_{jk}$, such choices made
independently, then the probability of any graph depends only
on its degree sequence and, moreover, the expected degree
sequence is~$\dvec$.
Conversely, the equal-probability condition implies that the
edge probabilities are related as in~\eqref{lameqn} and
the expected degree condition implies that~\eqref{saddlepoint}
holds~\cite{Csiszar}.
Following~\cite{CDS}, we call this the \textit{$\beta$-model} 
of random graph corresponding to~$\dvec$.

For the basic bipartite case $H=(\emptyset,\tilde E)$,
for any solution $\beta_1,\ldots,\beta_n$ and any
$b\in\Reals$, 
$\beta_1-b,\ldots,\beta_{n_1}-b,\beta_{n_1+1}+b,\ldots,\beta_n+b$
is also a solution, but note that the resulting values of
$\{p_{jk}\}$ remain the same.  With this caveat, the
solution exists and is unique if $\dvec$ lies in the
relative interior of the polytope of bipartite degree
sequences~\cite{RinaldoSUPP}.
Similarly to before, if we generate
a random bipartite graph with parts $V_1,V_2$ and
edges chosen independently with probabilities
$p_{jk}$, then the probability of every bipartite graph
with parts $V_1,V_2$ depends only on its degree sequence
and the expected degree sequence is~$\dvec$.
We will call
this the \textit{bipartite $\beta$-model} and note that it
is also called the \textit{Rasch model}~\cite{RinaldoSUPP}.

\smallskip
In the following subsections we will determine asymptotic
values for $N_H(\dvec)$ using
the same range of degree sequences as allowed by
Barvinok and Hartigan~\cite{BarvHart1}, but with non-trivial~$H$.
This will enable us to prove that the distribution of edges
within a constant or slowly-increasing set of vertex pairs
is asymptotically equal to that for the corresponding
$\beta$-model.  This strengthens the result of Chatterjee
et al.~\cite{CDS} that graphs with given degrees converge
in the sense of graph limits to the graphon defined by the
$\beta$-model, under some simple conditions.

In Section~\ref{S:concentration}, we show that the number
of edges within an arbitrary set of vertex pairs is concentrated
near the same value for random graphs with
given degrees and random graphs in the corresponding $\beta$-model.
This considerably strengthens similar results of Barvinok
and Hartigan~\cite{Barv01,BarvHart1}

In all cases, we will not present the best results our theory
allows so as to keep this example focussed.
As we described in the Introduction, the overall calculation consists
of several important steps of which the estimation of integrals
in the neighbourhoods of concentration points is the one
this paper is concerned with.
For the other steps, we will rely on the results of~\cite{BarvHart1}.
We will say more about that at the end of Section~\ref{S:bipartite}.

\subsection{General graphs}\label{S:subgraphs}

Throughout this subsection, we will define $\lambda_{jk}$ to be the
value of $p_{jk}$ in the solution of~\eqref{lameqn} subject
to~\eqref{saddlepoint} in the case $H=\Hempty$.

We now follow Barvinok and Hartigan~\cite{BarvHart1}
by requiring that $\dvec$ is
\textit{$\delta$-tame} for some $\delta>0$, which means that
$\delta\le\lambda_{jk}\le1-\delta$ for all $j\ne k$.
Chatterjee et al.~\cite{CDS} showed that $\delta$-tameness
follows if $\dvec$ is not too close to the boundary of the
Erd\H os-Gallai polytope.
Barvinok and Hartigan provide a useful sufficient condition.

\begin{lemma}[{\cite{BarvHart1}}]\label{tamesuff}
Let $0<\alpha<\beta<1$ satisfy $(\alpha+\beta)^2<4\alpha$.
Then if $\alpha(n-1)<d_j<\beta(n-1)$ for $1\le j\le n$ and
$n$ is large enough, there is some $\delta>0$ such that
$\dvec$ is $\delta$-tame. \qedhere
\end{lemma}

Define the $n\times n$ symmetric matrix $A$ by
\[
\thetavec\trans \!A\thetavec
  =\dfrac12\,\sum_{j<k} \lambda_{jk}(1-\lambda_{jk}) (\theta_j+\theta_k)^2.
\]
For each $j$, let $s_j$ be the number of times vertex~$j$ occurs in
$H^+\cup H^-$ and define $\smax=\max_{j=1}^n s_j$, $S=\tfrac12\sum_{j=1}^n s_j$
and $S_2=\sum_{j=1}^n s_j^2$.
Also define the following function, which arises from Taylor 
expansion of $G_{\dvec,H}(\thetavec)+\thetavec\trans\!A\thetavec$
about the origin.
\begin{equation}\label{fdef}
\begin{aligned}
    f_H(\thetavec) &= 
           i \sum_{\{j,k\}\in H^+} (1-\lambda_{jk}) (\theta_j+\theta_k)
           - i \sum_{\{j,k\}\in H^-}\lambda_{jk} (\theta_j+\theta_k) \\
          &{\quad}
           + \dfrac12  \sum_{\{j,k\}\in H^+\cup H^-} 
              \lambda_{jk}(1-\lambda_{jk})(\theta_j+\theta_k)^2 \\
          &{\quad}- \dfrac16 \,i \sum_{\{j,k\}\notin H^+\cup H^-}
                       \lambda_{jk}(1-\lambda_{jk})(1-2\lambda_{jk})
                        (\theta_j+\theta_k)^3 \\
          &{\quad}+ \dfrac{1}{24}\sum_{\{j,k\}\notin H^+\cup H^-}
                       \lambda_{jk}(1-\lambda_{jk})
                         (1-6\lambda_{jk}+6\lambda_{jk}^2)
                        (\theta_j+\theta_k)^4.
\end{aligned}
\end{equation}

Now we can state our main enumeration result, and the
resulting estimate of
\begin{equation*}
P_H(\dvec) = \frac{N_H(\dvec)}{N_{\Hempty}(\dvec)},
\end{equation*}
which is the probability that a uniform random graph with
degrees $\dvec$
contains $H^+$ and is disjoint from $H^-$. 

\begin{thm}\label{lmodel}
Let $\dvec$ be $\delta$-tame for some $\delta>0$.
Define $\{r_j\}, \{\lambda_{jk}\}, A, \smax, S_2, f_H$ as above,
and suppose that $\smax\le c_1 n^{1/6}$ and $S_2\le c_2 n$
for constants $c_1,c_2$.
Let $\X$ be a random variable with the normal density 
$\pi^{-n/2} \abs{A}^{1/2} e^{-\xvec\trans\!A\xvec}$.
Then,  for any $\eps>0$, there is a constant $c=c(\delta,\eps,c_1,c_2)$
such that
 \begin{equation}\label{NH}
  N_H(\dvec) = 
   2\, \pi^{n/2} C_{\dvec,H}\,\abs{A}^{-1/2} 
            e^{\E\Re f_H(\X)-\frac12\E(\Im f_H(\X))^2}(1+K),
 \end{equation}
 where $\abs{K}\le e^{c(1+\smax^3)n^{-1/2+\eps}}-1$.
 Moreover,
 \[
    P_H(\dvec) = (1+K') \prod_{\{j,k\}\in H^+} \negthickspace \lambda_{jk}
               \prod_{\{j,k\}\in H^-}  \negthickspace(1-\lambda_{jk})
 \]
 where $\abs{K'}\le e^{cS_2/n + c(1+\smax^3)n^{-1/2+\eps}}-1$.
\end{thm}

Note that formula \eqref{NH} in the case of $H=\Hempty$
matches~\cite[Thm.~1.4]{BarvHart1} apart from the error term.
The formula for $P_H(\dvec)$, absent the error term,
is the same as for the $\beta$-model.
The formula for $P_H(\dvec)$ is given more precisely in~\cite{ranx},
but only for the near-regular degree sequences considered there.
It considerably 
strengthens~\cite[Thm.~1]{CDS}, at least for $\delta$-tame
degree sequences.

\begin{proof}
For the duration of the proof, the implied constant in each
$O(\,)$ expression depends only on $\delta,\eps,c_1,c_2$.
We begin with a sequence of lemmas.
Define $\varOmega=U_n(\log n/n^{1/2})$.
\begin{lemma}\label{boxing}
For any $k>0$,
\[
  \int_{U_n(\pi)} G_{\dvec,H}(\thetavec)\,d\thetavec
  = 2\int_\varOmega G_{\dvec,H}(\thetavec)\,d\thetavec
   + O(n^{-k}) \int_\varOmega 
     \abs{G_{\dvec,\Hempty}(\thetavec)}\,d\thetavec
\]
\end{lemma}
\begin{proof}
 This is proved by the same method used in~\cite[Thm.~8.1]{BarvHart1},
with only a small change
 in their Lemma~8.4 to allow for the $o(n)$ factors for each~$j$, of the
 form $\abs{1+\lambda_{jk}(e^{i(\theta_j+\theta_k)}-1)}$,
 that appear in $\abs{G_{\dvec,\Hempty}}$ but not in
 $\abs{G_{\dvec,H}}$.
\end{proof}

\begin{lemma}\label{norms}
 Let $D$ be the diagonal matrix with the same diagonal as~$A$.
 Then for some constant $a_1$ we have 
 $\maxnorm{A^{-1}-D^{-1}} \le a_1 n^{-2}$.
 Furthermore, there exists a matrix $T$ 
 with $T\trans\! A T=I$ and some constants $a_2,a_3$ such that
 $\norm{T}_1,\norm{T}_\infty\le a_2 n^{-1/2}$
 and $\norm{T^{-1}}_\infty\le a_3n^{1/2}$.
\end{lemma}
\begin{proof}
 From the definition of $A$ we have $\maxnorm{A-D}\le \tfrac18$.
 Also, for any $\xvec$ we have
 \[
    \xvec\trans \!A\xvec \ge \tfrac12 \delta(1-\delta)\sum_{j<k} (x_j+x_k)^2
 \ge\tfrac12 \delta(1-\delta)(n-2)\xvec\trans\xvec,
\] 
 where we used the fact that the least eigenvalue of the matrix of the quadratic form $\sum_{j<k} (x_j+x_k)^2$ is~$n-2$.
  Taking into account that
  \begin{align*}
  \max D_{jj}\leq  \tfrac12 \delta(1-\delta) (n-1) \leq \tfrac18(n-1)  \text{~~and~~}\min D_{jj}\geq  \tfrac12 \delta(1-\delta) (n-1),
  \end{align*} 
  we apply Lemma~\ref{diagonal} with
 $r = n/\(4\delta(1-\delta)(n-1)\)$ and
 $\gamma= 4 \delta(1-\delta) \frac{n-2}{n-1}$ to complete the proof. 
\end{proof}

\begin{lemma}\label{expvar}
 We have
 \begin{align*}
    \E f_H (\X) &= \E f_{\Hempty} (\X)+ O(S/n) = O(1), \\
    \Var\Re f_H (\X)&= \Var\Re f_{\Hempty}(\X) + O(S_2/n^2) = O(1/n), \text{~~and} \\
    \Var\Im f_H (\X) &= \Var\Im f_{\Hempty} (\X)+ O(S_2/n) = O(1). 
 \end{align*}
\end{lemma}
\begin{proof}
 Consider the covariance matrix $(2A)^{-1}=(\sigma_{jk})$ and
 the random variable $\X=(X_1,\ldots,X_n)$ defined in the theorem.
 By Lemma~\ref{norms}, we have $\sigma_{jj}=O(n^{-1})$
 and $\sigma_{jk}=O(n^{-2})$ for all~$j\ne k$.  Lemma~\ref{moments}
 now tells us that any odd monomial in $X_1,\ldots,X_n$
 has mean~0, and that for $p,q\in\Naturals$ and $j\ne k, j'\ne k'$,
 \begin{align*}
   \E (X_j+X_k)^{2p} &= O(n^{-p});  \\
   \Cov\((X_j+X_k)^{2p+1},(X_{j'}+X_{k'})^{2q+1}\) &=
      \begin{cases} O(n^{-p-q-1}), & \text{~if~} \{j,k\}\cap\{j',k'\} \ne\emptyset \\
                             O(n^{-p-q-2}), & \text{~otherwise};
      \end{cases} \\[0.5ex]
   \Cov\((X_j+X_k)^{2p},(X_{j'}+X_{k'})^{2q}\) &=
         \begin{cases} O(n^{-p-q}), & \text{~if~} \{j,k\}\cap\{j',k'\} \ne\emptyset \\
                                O(n^{-p-q-2}), & \text{~otherwise}.
         \end{cases}
  \end{align*}
  Each of these is an obvious consequence of Lemma~\ref{moments}
  except perhaps the last claim. Consider monomials of the
  form $\mu\mu'$, where $\mu$ is a monomial in $\theta_j,
  \theta_k$ and $\mu'$ is a monomial in $\theta_{j'},\theta_{k'}$
  Pairings of the terms of $\mu\mu'$ which consist of a pairing
  of the terms of $\mu$ together with a pairing of the terms
  of $\mu'$ occur with the same constant in both
  $\E\((X_j+X_k)^{2p}(X_{j'}+X_{k'})^{2q}\)$
  and $\E(X_j+X_k)^{2p}\;\E (X_{j'}+X_{k'})^{2q}$.
  Because both $\mu$ and $\mu'$ are even, any other pairing
  of the terms of $\mu\mu'$ contains at least two of 
  $\sigma_{jj'},\sigma_{jk'},\sigma_{kj'},\sigma_{kk'}$, so\
  its value is at most $O(n^{-p-q-2})$. 
 
  Now we can just apply these bounds to the definition of $f_H$.
  It helps to use the fact that for real random variables
  $X_1,\ldots,X_m$ we have $\Var\(\sum_{j=1}^m X_j\)=
  \sum_{j,k=1}^m \Cov(X_j,X_k)$.
\end{proof}

Now we can complete the proof of Theorem~\ref{lmodel}
by applying Theorem~\ref{gauss4pt} to estimate
$\int_\varOmega G_{\dvec,H}(\thetavec)\, d\thetavec$.
  From Remark~\ref{boxremark} and the
 norm bound in Lemma~\ref{norms}, we can take $\rho_1=a_2^{-1}\log n$
 and $\rho_2=a_3\log n$.
  For $\thetavec\in T(U_n(\rho_2))$, we have by Taylor's theorem that
  \[ 
     G_{\dvec,H}(\thetavec) = e^{-\thetavec\trans\! A\thetavec
             + f_H(\thetavec) + h(\thetavec)},
  \]
  where $h(\thetavec)=O(n^{-1/2}(\log n)^5)$.
 
  From the definition of $f_H$ we find for
  $\thetavec\in T(U_n(\rho_2))$ that
  $\partial f_H/\partial \theta_j=O(\smax+(\log n)^2)$ for all~$j$.
  Similarly, for $j\ne k$,
  $\partial^2\! f_H/\partial\theta_j\partial\theta_k=O(1)$ if
  $\{j,k\}\in H^+\cup H^-$ and $O(n^{-1/2}\log n)$ otherwise.
  Finally, $\partial^2\! f_H/\partial\theta_j^2=O(n^{1/2}\log n)$
  for all~$j$.
  From the last two bounds we have
  $\norm{H(f_H,T(U_n(\rho_2)))}_\infty=O(n^{1/2}\log n)$.
  This gives us a value of $\phi_1=O(\smax n^{-1/6} \log n)$.
  
  The function $g$ in Theorem~\ref{gauss4pt} can be taken to
  be $\Re f_H$, whose first derivatives are bounded by
  $O(n^{-1/3}\log n)$ and Hessian by
  $\norm{H(\Re f_H,T(U_n(\rho_2)))}_\infty=O(n^{1/6})$.
  This gives a value of $\phi_2=O(n^{-1/2+\eps})$.

  We now find that all the conditions of Theorem~\ref{gauss4pt}  are
  satisfied. Apply Lemma~\ref{expvar} using
  $\V f_H=\Var\Re f_H -\Var\Im f_H = \Var\Re f_H-\E\,(\Im f_H)^2$,
  since $\E\Im f_H=0$.
  Finally, apply Lemma~\ref{boxing} with $k=1$.
  To estimate $\int_\varOmega\,\abs{G_{\Hempty}}$ use the same
  arguments as above using $\Re f_{(\emptyset,\emptyset)}(\thetavec)$
  in place of $f_H(\thetavec)$.
  This gives an added error term that fits into~$K$.
  Note that our conditions on $\smax$ allow for $K=-1$, but 
  even in that case the theorem is valid and gives a useful
  upper bound.  Finally, we can perform the
  division $N_H(\dvec)/N_{\Hempty}(\dvec)$ to 
  obtain $P_H(\dvec)$, noting that for the denominator
  the error term~$K$ is~$o(1)$.
  \end{proof}
  
  As we will demonstrate in Subsection~\ref{S:concentration}, 
  for obtaining concentration results it is worth noting that 
  the same method gives an upper bound for larger subgraphs.
  
  \begin{thm}\label{probbound}
    Let $\dvec$ be $\delta$-tame for some $\delta>0$.
   Define $\{\lambda_{jk}\}, S, \smax$ as above,
   and suppose that $\smax\le b_1 n^{2/3}/(\log n)^2$
   and $S\le b_2 n$
   for some constants $b_1,b_2>0$.
   Then there is $c'=c'(\delta,b_1,b_2)$ such that
 \[
    P_H(\dvec) \le c' \prod_{\{j,k\}\in H^+} \negthickspace \lambda_{jk}
               \prod_{\{j,k\}\in H^-}  \negthickspace(1-\lambda_{jk}).
 \]
  \end{thm}
  \begin{proof}
    The proof is the same as for Theorem~\ref{lmodel} except that we
    bound $N_H(\dvec)$ by using $\abs{G_{\dvec,H}(\thetavec)}$ in place of
    $G_{\dvec,H}(\thetavec)$.  This corresponds to dropping the imaginary
    parts of $f_H$.
    
    If $g(\theta)=\Re f_H(\theta)$ and $\thetavec\in T(U_n(\rho_2)$,
    then
     $\E g(\X)=O(1)$, $\Var g(\X)=O(n^{-1/3})$,
    $\abs{g_j}=O\((\smax+(\log n)^2)\log n/n^{1/2}\)$, and
    $\norm{H(g,T(U_n(\rho_2))}_\infty=O\(\smax+(\log n)^2\)$, 
   where in each case the implied constant depends only on
    $\delta,b_1,b_2$.
    Applying Theorem~\ref{gauss4pt} as before gives the theorem.
  \end{proof}
  
 
\nicebreak
\subsection{Bipartite graphs}\label{S:bipartite}

Define $V_1,V_2,n_1,n_2,\tilde E$ as before.  To keep the notation
parallel to the notation in the previous section, we will assume that
$\tilde E, H^+, H^-$ are disjoint.

This case is not covered by the previous subsection
since the set of forbidden edges is too big for Theorems~\ref{lmodel}
and~\ref{probbound}. 
Nevertheless,  we will derive similar results by using formula \eqref{NHd} 
with the radii chosen in such a way that the contours
pass through the saddle point for $H =(\emptyset, \tilde{E})$.
Accordingly, let $\tilde\lambda_{jk}$ be the
value of $p_{jk}$ in the solution of~\eqref{lameqn} subject
to~\eqref{saddlepoint} in the case $H=(\emptyset,\tilde E)$.

Define
$\tilde N_H(\dvec)=N_{(H^+,H^-\cup\tilde E)}(\dvec)$, 
$\tilde P_H(\dvec)=P_{(H^+,H^-\cup\tilde E)}(\dvec)$, and 
$\tilde C_H(\dvec)=C_{(H^+,H^-\cup\tilde E)}(\dvec)$.
Thus $\tilde N_H(\dvec)$ is the number of bipartite graphs
with degrees $\dvec$, on $(V_1,V_2)$ that contain
$H^+$ and are disjoint from $H^-$, and 
$\tilde P_H(\dvec)$ is the fraction of such graphs among
all bipartite graphs on $(V_1,V_2)$ with degrees~$\dvec$.

Define $\tilde G_{\dvec,H}(\thetavec)=G_{\dvec,(H^+,H^-\cup\tilde E)}(\thetavec)$
and $\tilde f_H(\thetavec)=f_{(H^+,H^-\cup\tilde E)}(\thetavec)$ as
in~\eqref{Gdef} and~\eqref{fdef}, but using $\{\tilde\lambda_{jk}\}$
instead of~$\{\lambda_{jk}\}$.

With a tiny adjustment,
we adopt from Barvinok and Hartigan~\cite{BarvHart1} conditions
on $\dvec$ that we call \textit{$\delta$-bitame} for $\delta>0$:
$\delta \leq \tilde\lambda_{jk} \leq 1-\delta$ for all $\{j,k\} \notin \tilde{E}$
and $n_1,n_2\ge \delta n$.
We also provide a sufficient condition similar to Lemma~\ref{tamesuff}.
\begin{lemma}\label{bitamestuff}
   Let $p,q$ be real numbers such that $0<q^2< p\leq q<1$. 
   Then for  any degree sequence $d_1, \ldots, d_n$ such that 
   $ \sum_{j \in V_1} d_j = \sum_{j \in V_2} d_j$ and 
  \[
   pn_2 \leq d_j \leq q n_2  \ \text{ for }\  j\in V_1, \qquad
   pn_1 \leq d_j \leq q n_1 \  \text{ for } \ j\in V_2,
  \]
   the solution $\{\tilde\lambda_{jk}\}$ defined above exists
    and $\delta < \tilde\lambda_{jk} < 1-\delta$ for all $\{j,k\}\notin \tilde E$, 
   where $\delta>0$ depends only on~$p,q$.
\end{lemma}
\begin{proof}
 		Without loss of generality we can assume that $d_1\geq \cdots \geq d_{n_1}$.
 		To prove the existence of the solution $\{\tilde\lambda_{jk}\}$ if will
 		suffice to show that all the Gale-Ryser inequalities are strict~\cite{RinaldoSUPP};
 		i.e., for any $1 \leq k< n_1$,
 		\[
 			\sum_{j=1}^k d_j < \sum_{j\in V_2} \min\{d_j,k\}.
 		\]
 		If $q n_1<k<n_1$ then $\sum_{j \in V_2} \min\{d_j,k\} = \sum_{j \in V_2} d_j > \sum_{j=1}^k d_j$. 
 		For $k < p n_1$ we get that
 		$\sum_{j\in V_2} \min\{d_j,k\} = k n_2 > k q n_2 \ge \sum_{j=1}^k d_j$.   
 		For the remaining case, when $p n_1\leq k \leq q n_1$, observe that
 		$\sum_{j\in V_2} \min\{d_j,k\} \geq p n_1 n_2 > q^2 n_1 n_2 \geq kq n_2$.
 		
 		If $\{\beta_j\}$ are the parameters in~\eqref{lameqn} corresponding
 		to $\{\tilde\lambda_{jk}\}$, and $c$ is a constant, recall that 
 		$\beta_1-c,\ldots,\beta_{n_1}-c,\beta_{n_1+1}+c,\ldots,\beta_n+c$
 		is also a solution.
 		By choice of $c$, we can assume for some $\gamma \in [0,1]$ that
 		\begin{equation}\label{assumption_V+-}
 		   \card{V_1^+} \geq \gamma n_1, \ \ \card{V_1^-}\geq (1-\gamma)n_1,
		    \ \  \card{V_2^+} \geq \gamma n_2,
		     \ \ \card{V_2^-}\geq (1-\gamma)n_2,
 		\end{equation}
 		 where $V_t^{\pm} =\{j \in V_t \st \pm \beta_j \geq 0\}$.  
 		 Recalling  that for $\{j,k\} \notin \tilde E$
 		 \[
 		 	\tilde \lambda_{jk} = \frac{e^{\beta_j+\beta_k}} {1+e^{\beta_j+\beta_k}}, 
 		 \]
 		it is sufficient to show that $\abs{\beta_j} \leq b$, $j=1,\ldots,n$
 		 for some $b=b(p,q)>0$.
 
 		 Define $a = \max_{j\in V_1} {\beta_j}$ and
 		 		 		 		 $b = \min_{j\in V_2} {\beta_j}$.
 		 Without loss of generality, we can assume that $a=\beta_1$ and $b=\beta_n$.		 		 
 		 Note that 
 		 \[
 		 	d_1 = \sum_{j\in V_2} \tilde\lambda_{1j} \geq n_2  \frac{e^{a+b}} {1+e^{a+b}},
 		 	\qquad 
 		 	d_n = \sum_{j\in V_1} \tilde\lambda_{jn} \leq n_1  \frac{e^{a+b}} {1+e^{a+b}}. \ \ \ 
 		 \]	 
 		 By assumption $d_1 \leq q n_2$ and $d_n \geq p n_1$, therefore
 		 \begin{equation}\label{bound(a+b)}
 		 	p\leq \frac{e^{a+b}} {1+e^{a+b}} \leq q  \ \ \Longrightarrow  \ \ \log \frac{p}{1-p} \leq a+b \leq \log \frac{q}{1-q}.
 		 \end{equation}
 		  Using \eqref{assumption_V+-}, we find also that
 		  \[
 		  		d_1 =\sum_{j\in V_2} \tilde\lambda_{1j} \geq \gamma n_2 \frac{e^{a}} {1+e^{a}}, \ \ \ \ 
 		  			d_n =\sum_{j\in V_1} \tilde\lambda_{jn} \leq 	\gamma n_2  \frac{e^{a+b}} {1+e^{a+b}} +  (1-\gamma)n_2\frac{e^{b}} {1+e^{b}},
 		  \]
 		  which gives us $q \geq \gamma  \frac{e^{a}} {1+e^{a}}$
 		  	and $p \leq  \gamma \frac{e^{a+b}} {1+e^{a+b}} +  (1-\gamma)\frac{e^{b}} {1+e^{b}} \leq \gamma q +(1-\gamma)\frac{e^{b}} {1+e^{b}}$. 
 		  			  	If $\gamma \geq (p+q^2)/2q$ then the first  inequality implies
 		  			  	 $\frac{e^{a}} {1+e^{a}} \leq 2q^2/(p+q^2)$.
 		  			  	  Otherwise, from the second inequality we get that $\frac{e^{b}} {1+e^{b}} \geq q(p-q^2)/(2q-p-q^2)$. Using \eqref{bound(a+b)}, 
 		  			  	we get in the both cases that  
 		  			  	\[
 		  			  	   \max_{j\in V_1} {\beta_j} =a \leq b \ \ \ \text{ and } \ \ \
 		  			  	    \min_{j\in V_2} {\beta_j} = b\geq - b  \ \ \ \ \text{ for some $b=b(p,q)>0$.}
 		  			  	\]
 		  		In order to get the missing reverse bounds and to complete the proof,  we just need  to swap the roles of subsets $V_1$, $V_2$. 	  	
\end{proof}

Define the $n\times n$ symmetric matrix $\tilde{A}$ by 
 \[
 	\thetavec\trans\! \tilde{A} \thetavec = \dfrac12 \sum_{\{j,k\} \notin \tilde{E}}
	 \tilde \lambda_{jk}(1-\tilde \lambda_{jk})(\theta_j +\theta_k)^2.
 \]
 For each $j$, let $s_j$ be the number of times vertex~$j$ occurs in
$H^+\cup H^-$ and define
$\smax=\max_{j=1}^n s_j$, $S=\tfrac12\sum_{j=1}^n s_j$
and $S_2=\sum_{j=1}^n s_j^2$.
Differently from the matrix $A$ in the previous subsection,
$\tilde A$ has a zero eigenvalue.
 Let $\wvec = (w_1,\ldots, w_n)\trans$ be defined by
 $w_j = (-1)^m$ if $j\in V_m$ for $m=1,2$. 
 Note that $\ker \tilde{A}=\langle\wvec\rangle$  and
 $\tilde{f}(\thetavec + t\wvec ) = \tilde{f}(\thetavec)$ 
  for any $t\in \Reals$ and $\thetavec \in \Reals^n$. 

\begin{thm}\label{blmodel}
Let $\dvec$ be $\delta$-bitame for some $\delta>0$.
Define $\{\tilde \lambda_{jk}\}, \tilde A,s_{\max}, S_2, \tilde f_H, \wvec$ as above,
and suppose that $s_{\max}\le c_1 n^{1/6}$ and $S_2\le c_2n$
for some constants $c_1,c_2$.
Let $\tilde\X$ be a random variable with the normal density 
$\pi^{-n/2} \abs{\tilde A + \wvec\wvec\trans}^{-1/2}
  e^{-\xvec\trans\!(\tilde A + \wvec\wvec\trans)\xvec}$.
Then, for any $\eps>0$, there is a constant $\tilde{c}=\tilde{c}(\delta,\eps,c_1,c_2)$ such that
 \begin{equation}\label{bNH}
  N_H(\dvec) = 
    2\pi^{(n+1)/2}\, n \,\tilde C_{\dvec,H} \,\abs{\tilde A+ \wvec\wvec\trans}^{-1/2}\,
            e^{\E\Re \tilde f_H(\tilde\X)-\frac12\E(\Im \tilde f_H(\tilde \X))^2}(1+\tilde K),
 \end{equation}
 where $\abs{\tilde K}\le e^{\tilde c(1+s_{\max}^3)n^{-1/2+\eps}}-1$.
 Moreover,
 \[
    \tilde P_H(\dvec) =
    (1+\tilde K') \prod_{\{j,k\}\in H^+} \negthickspace \tilde \lambda_{jk}
               \prod_{\{j,k\}\in H^-}  \negthickspace(1-\tilde \lambda_{jk})
 \]
 where $\abs{\tilde K'}\le 
   e^{\tilde c S_2/n + \tilde{c}(1+s_{\max}^3)n^{-1/2+\eps}}-1$.
\end{thm} 

Using Corollary \ref{C:expectations} one can note that~\eqref{bNH}
in the case of $H=\Hempty$ (with a different error
term) matches~\cite[formula (2.5.4)]{BarvHart1}.
The formula for $\tilde P_H(\dvec)$ is given more precisely in~\cite{GMX},
but only for the near-semiregular degree sequences considered there.
 
\begin{proof}
We start from formula \eqref{NHd}. 
Since  $\tilde G_{\dvec,H}(\thetavec + t\wvec ) = \tilde G_{\dvec,H}(\thetavec)$ 
 for any $t\in \Reals$ and $\thetavec \in \Reals^n$
we can fix $\theta_n =0$ and multiply by $2\pi$ to obtain
\[
  \tilde N_H(\dvec) = 2\pi\,\tilde C_{\dvec,H}
      \int_{U_{n-1}(\pi)} \tilde G_{\dvec,H}(\thetavec)\,d\thetavec',
\]
where $\thetavec = \thetavec(\thetavec') = (\theta_1', \ldots, \theta_{n-1}',0)$ .
Let $\varOmega=U_n(\log n/n^{1/2})$ and 
$L = \{\thetavec \in \Reals^n \st \theta_n =0\}$. 

\begin{lemma}\label{bboxing}
For any $k>0$,
\[
  \int_{U_{n-1}(\pi)} \tilde G_{\dvec,H}(\thetavec)\,d\thetavec'
  = \int_{\varOmega\cap L} \tilde G_{\dvec,H}(\thetavec)\,d\thetavec'
   + O(n^{-k}) \int_{\varOmega\cap L}
     \abs{\tilde G_{\dvec,\Hempty}(\thetavec)}\,d\thetavec'
\]
\end{lemma}
\begin{proof}
  This follows from~\cite[p.~340]{BarvHart1} in the same way that 
  Lemma~\ref{boxing} follows from~\cite[Thm.~8.1]{BarvHart1}.
  Note that Barvinok and Hartigan do not actually provide a proof,
  but we agree with them that there is a proof
  parallel to that of their Theorem~8.1.
\end{proof}

Define matrices $Q,W,P,R$ by
\begin{align*}
	Q\xvec &= \xvec - x_n \wvec , \qquad W \xvec = \tfrac{1}{\sqrt{n}}\wvec  \wvec \trans \xvec,\\
	 P \xvec &= \xvec  - \tfrac{1}{n} \wvec  \wvec \trans \xvec , \qquad
	 R \xvec  =\tfrac{1}{\sqrt{n}}\xvec.
\end{align*}
Applying Lemma \ref{LemmaQW} with $\rho =  \log n$, we find that
\begin{align*}
	\int_{\varOmega\cap L} \tilde G_{\dvec,H}(\thetavec)\,d\thetavec' &= 
	\(1+ O(n^{-\log n}) \) \pi^{-1/2} n
	  \int_{\varOmega_\rho} \tilde G_{\dvec,H}(\thetavec)\,
	     e^{-\thetavec\trans \wvec\wvec \trans   \thetavec  }\,d\thetavec,\\
	\int_{\varOmega\cap L}
     \abs{\tilde G_{\dvec,\Hempty}(\thetavec)}\,d\thetavec
      &= 	\(1+ O(n^{-\log n}) \) \pi^{-1/2} n \int_{\varOmega_\rho}
     \abs{\tilde G_{\dvec,\Hempty}(\thetavec)} \,
          e^{-\thetavec\trans \!\wvec\wvec \trans   \thetavec  }\,d\thetavec,
\end{align*}
and also that $ U_n(\tfrac12  \log n /n^{1/2} ) \subseteq \varOmega_\rho \subseteq U_n(3  \log n /n^{1/2} )$.

We continue  the proof of Theorem~\ref{blmodel} with a sequence of lemmas.

\begin{lemma}\label{bnorms}
 Let $D$ be the diagonal matrix with the same diagonal as~$\tilde A$.
 Then for some constant $a_1$ we have 
 $\maxnorm{(\tilde A+\wvec\wvec\trans)^{-1}-D^{-1}} \le a_1 n^{-2}$.
 Furthermore, there exists a matrix $T$ 
 with $T\trans (\tilde A+\wvec\wvec\trans) T=I$ and some constants $a_2,a_3$ such that
 $\norm{T}_1,\norm{T}_\infty\le a_2 n^{-1/2}$
 and $\norm{T^{-1}}_\infty\le a_3n^{1/2}$.
\end{lemma}
\begin{proof}
 From the definition of $\tilde{A}$ we have $\maxnorm{\tilde{A}-D}\le \tfrac18$.
 Also, for any $\xvec$ such that $\wvec\trans \xvec = 0$ we have
 \[
   \xvec\trans  (\tilde A+ \wvec\trans\wvec) \xvec = \xvec\trans \!\tilde A\xvec \ge \tfrac12 \delta(1-\delta)\sum_{\{j,k\}\in \tilde E} (x_j+x_k)^2
 \ge\tfrac12 \delta(1-\delta) \delta n \xvec\trans\xvec,
 \]
 where we used the fact 
 that all eigenvalues with exception of one zero (which  corresponds
 to $\wvec$) of the quadratic form 
  $\sum_{\{j,k\}\notin \tilde E}\, (x_j+x_k)^2$ 
  are at least $\min\{ \card{V_1}, \card{V_2}\} \geq \delta n$.
   We note also that $\maxnorm{\wvec\trans\wvec} =1$ and for any $\xvec = t \wvec$
 $$\xvec\trans  (\tilde A+ \wvec\trans\wvec) \xvec = \xvec\trans \wvec\trans\wvec \xvec = n \xvec\trans \xvec.$$ 
 Taking into account the following inequalities: 
 \begin{align*}\max D_{jj} \leq  \tfrac12 \delta(1-\delta) \max\{ \card{V_1}, \card{V_2}\} \leq \tfrac18(1-\delta)n, \\
 	 \min D_{jj}  \geq \tfrac12 \delta(1-\delta) \min\{ \card{V_1}, \card{V_2}\} \geq \tfrac12 \delta^2 (1-\delta)n,
 \end{align*} 
  we finish the proof by applying Lemma~\ref{diagonal} with
 $r = 9/\(4\delta^2 (1-\delta)\)$ and
 $\gamma= 4\delta^2$. 
\end{proof}

\begin{lemma}\label{bexpvar}
 We have
 \begin{align*}
    \E \tilde f_H(\tilde{\X}) &= \E \tilde f_{\Hempty}(\tilde{\X}) + O(S/n) = O(1), \\
    \Var\Re \tilde f_H(\tilde{\X}) &= \Var\Re \tilde f_{\Hempty}(\tilde{\X}) + O(S_2/n^2) = O(1/n), \\
    \Var\Im \tilde f_H(\tilde{\X}) &= \Var\Im \tilde f_{\Hempty}(\tilde{\X}) + O(S_2/n) = O(1). 
 \end{align*}
\end{lemma}
\noindent
Lemma \ref{bexpvar} is proved in precisely the same way as  Lemma \ref{expvar}.

In order to estimate
$\int_{\varOmega_\rho} \tilde{G}_{\dvec,H}(\thetavec) \,
  e^{-\thetavec\trans \wvec\wvec \trans   \thetavec} d\thetavec$ 
  and $\int_{\varOmega_\rho} \abs{\tilde{G}_{\dvec,\Hempty}(\thetavec)}\,
    e^{-\thetavec\trans \wvec\wvec \trans \thetavec} d \thetavec$ we 
apply Theorem~\ref{gauss4pt}. 
  From Remark~\ref{boxremark} and the
 norm bound in Lemma~\ref{bnorms}, we can take $\rho_1=\tfrac 12 a_2^{-1}\log n$
 and $\rho_2=3 a_3\log n$.
   For $\thetavec\in T(U_n(\rho_2))$, we have by Taylor's theorem that
  \[ 
     \tilde G_{\dvec,H}(\thetavec) e^{-\thetavec\trans \wvec\wvec \trans   \thetavec} = e^{-\thetavec\trans (\tilde A +  \wvec\wvec \trans )  \thetavec
             + \tilde{f}_H(\thetavec) + \tilde{h}(\thetavec)},
  \]
  where $\tilde{h}(\thetavec)=O(n^{-1/2}(\log n)^5)$. Now the proof can be finished
   in complete analogy with the proof of Theorem \ref{lmodel}.
\end{proof}

 The same argument gives us also the analog of Theorem \ref{probbound} for bipartite case which will be useful for obtaining concentration results.
  
\begin{thm}\label{bprobbound}
    Let $\dvec$ be $\delta$-bitame for some $\delta>0$.
   Define $\{\tilde\lambda_{jk}\}, S, s_{\max}$ as above,
   and suppose that $s_{\max}\le b_1 n^{2/3}/(\log n)^2$
   and $S\le b_2 n$
   for some constants $b_1,b_2>0$.
   Then there is $\tilde c'=\tilde c'(\delta,b_1,b_2)$ such that
 \[
    \tilde P_H(\dvec) \le \tilde c' \prod_{\{j,k\}\in H^+} \negthickspace \tilde \lambda_{jk}
               \prod_{\{j,k\}\in H^-}  \negthickspace(1-\tilde \lambda_{jk}).
 \]
  \end{thm}

\begin{remark}
 Theorems~\ref{lmodel} and~\ref{blmodel} are less general than our
 techniques allow, due to the choices that we made here for the
 purpose of keeping our example simple.
 We restricted ourselves to $\delta$-tame and $\delta$-bitame
 degree sequences so that we could adopt Lemmas~\ref{boxing}
 and~\ref{bboxing} from~\cite{BarvHart1}.
 More significantly, we used the saddle point of $f$ for $f_H$
 as well, which simplifies the calculation a lot at the expense of
 restricting $H$ far more than necessary. 
 In a follow-up paper, we will show how to estimate $N_H(\dvec)$
 whenever the quadratic form
 \[
     \sum_{jk\notin H^+\cup H^-} \lambda_{jk}(1-\lambda_{jk})(\theta_j+\theta_k)^2
 \]
  has not too many zero eigenvalues and all
 its nonzero eigenvalues are at least~$\delta n$, where
 $\delta>0$ may be constant or slowly decreasing. 
\end{remark}

  
  \nicebreak
  \subsection{Concentration near the $\beta$-model}\label{S:concentration}
 
  For a given degree sequence $\dvec$ and pair of vertices
  $j\ne k$ (in the bipartite case, for  $\{j,k\}\notin \tilde E$), 
  let $\xi_{jk}=\xi_{jk}(\dvec)$ be the indicator
  variable for $\{j,k\}$ being an edge in a uniformly random graph (or, alternatively, a uniformly random bipartite graph with  partite sets $V_1$, $V_2$)
  with degree sequence~$\dvec$.  Let $\{\hat\xi_{jk}\}$ be independent
  Bernoulli variables with $\Prob(\hat\xi_{jk}=1)=\lambda_{jk}$ for
  all $j,k$ (or, in the bipartite case, $\Prob(\hat\xi_{jk}=1)=\tilde\lambda_{jk}$
   for pairs $\{j,k\}\notin \tilde E$).
  Note that $\{\hat\xi_{jk}\}$ is just the $\beta$-model.
   
  Theorems~\ref{lmodel} and~\ref{blmodel} show that $\{\xi_{jk}\}$ and
  $\{\hat\xi_{jk}\}$ are point-wise almost identical at small scales.
  Now we explore their relationship at large scales.
  Let $Y$ be a set of vertex pairs (disjoint from $\tilde E$
  in the bipartite case).
  Define $X=X(Y,\dvec)=\sum_{jk\in Y} \xi_{jk}$ and 
  $\hat X= \hat X(Y,\dvec)= \sum_{jk\in Y} \hat\xi_{jk}$.
  From Theorems~\ref{lmodel} and~\ref{blmodel} we have that 
  $\E \hat X^t\sim \E X^t$ for $t=O(n^{1/6-\eps})$, but this is
  not sufficient to estimate $\Var \hat X$.
  
  Barvinok~\cite{Barv01}, in the bipartite case under conditions
  more general than $\delta$-bitameness, and Barvinok and
  Hartigan~\cite{BarvHart1} in the general case under $\delta$-tameness,
  show that for, $\abs Y \ge \delta n^2$,
  \begin{equation}\label{BVcon}
      (1-\delta n^{-1/2} \log n)\E \hat X \le X
     \le (1+\delta n^{-1/2} \log n)\E \hat X
 \end{equation}
 with probability~$1-n^{-\Omega(n)}$.
 In the case of near-regular degree sequences, McKay~\cite{ranx}
 proved a weaker concentration of $X$ near $\E \hat X$ whenever
 $\card{Y}\to\infty$.
 Note that~\eqref{BVcon} starts to ``bite'' at around
 $\card{Y}^{3/4}$ from the mean.
 Since the variance of $\hat X$ has the same order as the
 expectation of $\hat X$ for all~$Y$, it seems likely that a concentration
 inequality that bites at around $\card{Y}^{1/2}$ from the mean
 is the best that can be hoped for without specifying more
 structure for~$Y$.  Here we prove such concentration in both
 the general and bipartite cases, starting with a lemma
 that bounds the moments of~$X$ in terms of the moments
 of~$\hat X$.
 
 \begin{lemma}\label{mombound}
    Let the assumptions of Theorem \ref{Thm_concentr} hold.
    Then for any $b>0$,
    there is a constant $\hat c = \hat c(\delta,b)>0$  such that 
          $\E X^m \le \hat c \E \hat X^m$ for all integers $m$ with
          $0 \leq m\le b\,\card{Y}^{1/2}n^{1/6}/(\log n)^3$.
 \end{lemma}
   \begin{proof}
      For $1\le t\le m$, let 
      \[  X_t = \sum_{W\subseteq Y: \card{W}=t}\; \prod_{jk\in W} \xi_{jk},
      \text{~~and~~} 
          \hat X_t = \sum_{W\subseteq Y: \card{W}=t}\; \prod_{jk\in W} \hat\xi_{jk}.
      \]
      Since these are indicator variables, we have
      \[
          X^m = \sum_{t=1}^m t!\,\stirlingii mt X_t
          \text{~~and~~}
          \hat X^m = \sum_{t=1}^m t!\,\stirlingii mt \hat X_t,
      \]
      where $\stirlingii mt$ is the Stirling number of the second kind.
      It follows that the assertion will be true if
      $\E X_t\le \hat c\E \hat X_t$ for $1\le t\le m$, where $\hat c$ is
      a constant depending only on~$b$ and~$\delta$. 
      Due to Theorem~\ref{probbound}, we immediately get this bound if $m \leq b_1 n^{2/3}/(\log n)^2$ 
      (and, consequently, if $\card{Y}\leq (b_1/b)^2 n (\log n)^2$) for any fixed $b_1>0$. 
      For greater values of $m$ and $\card{Y}$, it requires additional consideration.
      
      Any subset $W\subseteq Y$ induces a graph on $n$ vertices. 
      Let $w_j$ denote the degree of $j$ in this graph. 
      We refer to a vertex $j$ as a $W$-\textit{full} vertex if $w_j > \lfloor n^{2/3}/(\log n)^2\rfloor$ and
      a pair $jk\in Y$ as a $W$-\textit{critical} pair  if  at least one of the
      vertices $j$, $k$ is  $W$-full.
      Define
      $\eta(W)=\sum_{j=1}^n \max\{0,w_j-\lfloor n^{2/3}/(\log n)^2\rfloor\}$.
      Since a set satisfying Theorem~\ref{probbound} is obtained by 
      removing at most $\eta(W)$ elements from~$W$, we have that
      \[
         \E\, \Bigl(\prod_{jk\in W} \xi_{jk}\Bigr) \le
          p(W), \text{~~where~~}
           p(W) = c'  \delta^{-\eta(W)} \prod_{jk\in W}\lambda_{jk},
      \]
      where $c'$ is the constant from Theorem~\ref{probbound}.
      Consequently, 
      \[
         \E X_t \le c'  \sum_{W\subseteq Y: \card{W}=t} p(W).
      \]
      We now apply Lemma~\ref{switching}, stated in the Appendix.
      Define a digraph $D$ whose vertices
      are the $t$-subsets of~$Y$.  The ordered pair $(W,W')$ is
      an edge of $D$ if 
      $W - W'$ consists of one element,  which is
      $W$-critical.
      Define $s,\alpha:E(D)\to \Reals$ by
      \[
         s(W,W') = \frac{p(W)p(W')}{\sum_{W'':(W'',W')\in E(D)} p(W'')}
         \text{~~and~~}
         \alpha(W,W') = 
            \frac{\sum_{W'':(W'',W')\in E(D)} p(W'')}
                   {\sum_{W'':(W,W'')\in E(D)} p(W'')}.
      \]
      It is routine to check that the conditions of Lemma~\ref{switching}
      are satisfied, with $Z$ being the set of vertices $W$ with
      $\eta(W)=0$, provided we have $\alpha(W,W')<1$ for every edge.
      
      Given $W$ with $\eta(W)>0$, we can choose a $W$-critical pair 
      belonging to $W$ in at least    $n^{2/3}/(\log n)^2$ ways, and then we can replace it  by some element of $Y-W$ in at least $\card{Y}-t$ ways.
      Thus, the out-degree of~$W$ is at least $n^{2/3}(\card{Y}-t)/(\log n)^2$.
      Alternatively, given $W'$, we can choose an element $jk\in W'$
       in $t$ ways,
      choose a vertex presented in at least $\lfloor n^{2/3}/(\log n)^2 \rfloor$ pairs 
      of $W'$ (i.e. $W'$-full or almost $W'$-full) in at most $t/ \lfloor n^{2/3}/(\log n)^2 \rfloor$ ways and replace~$jk$
      by a pair containing that vertex in at most~$n$ ways.
      So the in-degree of $W'$ is at most $t^2n/\lfloor n^{2/3}/(\log n)^2 \rfloor$.
      Finally, an in-neighbour $W_1$ of $W'$ differs in at most 3 elements
      from an out-neighbour $W_2$ of $W$, so
      $\abs{\eta(W_1)-\eta(W_2)}\le 6$.
      Since $t\le m$, we find that
      $\alpha(W,W')\le \hat\alpha=2b^2\delta^{-6}/(\log n)^2<\tfrac12$
      for large enough~$n$.
      
      Consequently, since $p(W)=c'\prod_{jk\in W} \lambda_{jk}$ when
      $\eta(W)=0$, Lemma~\ref{switching} tells us that
       \[
          \E X_t \le \frac{1-\hat\alpha}{1-2\hat\alpha} \sum_{W\subseteq Y : \eta(W)=0} p(W)
            \le c' \frac{1-\hat\alpha}{1-2\hat\alpha}
              \sum_{W\subseteq Y: \card{W}=t}\; \prod_{jk\in W} \lambda_{jk}
              = c' \frac{1-\hat\alpha}{1-2\hat\alpha} \E\hat X_t.
       \]
       This completes the proof.
   \end{proof}
  
 \begin{thm}\label{Thm_concentr}	Suppose $\dvec$ is $\delta$-tame 
 (or $\delta$-bitame) for some $\delta>0$. Let $Y$ be a set of vertex pairs 
 (disjoint from $\tilde E$ in the bipartite case). Then for any $\gamma>0$
\[
  P\( 
  	\abs{X  - \E\hat X} <  \gamma \card{Y}^{1/2} 
  	\) \geq 1 - \breve{c} e^{- 2 \gamma \min\{\gamma, n^{1/6}(\log n)^{-3} \} }, 
\]
  		where the constant $\breve{c}>0$ depends only on~$\delta$.
  \end{thm}

  \begin{proof}
    The proofs of the general and bipartite cases are the same;
    we will use the notation of the general case.
      Let $p>0$ be such that $p\card{Y} = \E \hat X   = \sum_{jk \in Y} \lambda_{jk}$. 
      	   Hoeffding's Lemma (see \cite{Hoeffding} and Lemma \ref{Hoeffding}) gives us for any $t>0$
\[
      	  	\E e^{t \hat X} \leq e^{tp \card{Y}+ \frac{1}{8}t^2\card{Y} }. 
\]
      Using Lemmas~\ref{mombound} and~\ref{series_exp},
	  we find that for $t\leq \frac{4n^{1/6}}{(\log n)^3 } \card{Y}^{-1/2}$
\[
      	  	\E e^{t X} \leq \dfrac43 
      	  	\sum_{k=0}^{\bigl\lfloor\frac{16 n^{1/6}\card{Y}^{1/2}}{(\log n)^3 } \bigr\rfloor}
		 \frac{ t^k\E X^k}{k!} \leq \dfrac43\,\hat{c} \E e^{t \hat X}
      	  	 \leq \dfrac43\,\hat{c} e^{tp\card{Y}+ \frac{1}{8}t^2\card{Y} }.
\]
      	  Taking $t = 4\card{Y}^{-1/2} \min\{ \frac{n^{1/6}}{(\log n)^3 }, \gamma\}$
	   and using  Markov's inequality for $e^{t X}$, we obtain that 
\[
  			P\( 
  			    X \geq p\card{Y} + \gamma \card{Y}^{1/2}
  			\) \leq  \frac{\E e^{tX}}{ e^{tp\card{Y} + t\gamma \card{Y}^{1/2}}} \leq \dfrac43\,\hat{c}  e^{- t\gamma \card{Y}^{1/2}+ \frac{1}{8}t^2\card{Y}}
  			\leq \dfrac43\,\hat{c}  e^{-2 \gamma \min\bigl\{\gamma, \frac{n^{1/2}}{(\log n)^3 }\bigr\} }. 
\]
      	   
      	   To complete the proof we apply the same arguments for the complement degree sequence
      	   $\dvec^c = (n-1-d_1, \ldots, n-1-d_n)$ which is also $\delta$-tame with $\lambda_{jk}^c = 1-\lambda_{jk}$
     \end{proof}

\nicebreak
\section{Appendix}

Here we give proofs of some technical results that were used
in the proofs.

\begin{lemma}\label{new_ineq}
If $z_1,z_2 \in \Complexes$ satisfy $\abs{z_1} \le \alpha$ and $\abs{z_2} \leq \beta$, then
 \begin{align*}
   \Abs{e^{z_1} - e^{{z_1^2}/2} - z_1} &\leq e^{\tfrac16 \alpha^3 + \tfrac18 \alpha^4} -1,
 \\
 \Abs{ z_1 (e^{z_2}-z_2-1)} 
       &\leq e^{\tfrac18 \beta^2} +  e^{\tfrac13 \alpha \beta + \tfrac14 \beta^2 + \tfrac14 \alpha^4} - \tfrac{1}{3} \alpha \beta - 2.
\end{align*}
\end{lemma}
\begin{proof}
From the signs of the Taylor coefficients of
$e^{z_1}-e^{z_1^2/2}-z_1$ we see that the left
side of the first inequality is largest
when $z_1=-\alpha$. This means we only
need to prove $\phi(\alpha)\ge 0$ for $\alpha\ge 0$,
where
\begin{equation*}
   \phi(\alpha) = e^{\alpha^3/6+\alpha^4/8}
    - e^{\alpha^2/2} + e^{-\alpha} + \alpha - 1.
\end{equation*}
It is clear that $\phi(\alpha)>0$ for
$\alpha>\tfrac32$, since in that case
$\tfrac16\alpha^3+\tfrac18\alpha^4>\alpha^2/2$
and $e^{-\alpha}>1-\alpha$.
For $0\le\alpha\le\tfrac32$ we can apply
$e^x \le 1+x+\tfrac12 x^2 + \tfrac16 x^3 + \tfrac1{18} x^4$
for $0\le x\le\frac 65$,
$e^x \le 1+x+\tfrac12 x^2 + \tfrac16 x^3$ for $x\ge 0$,
and
$e^{-x}\ge \sum_{i=0}^5 \tfrac{1}{i!}(-x)^i$ for
$x\ge 0$.  This gives us a polynomial of degree 12
that is nonnegative for all $\alpha\ge 0$ and
bounds $\phi(\alpha)$ from below.

For the second inequality, the worst case is
obviously $z_1=\alpha, z_2=\beta$, so
we just need to prove that $\varphi(\alpha,\beta)\ge 0$
for $\alpha,\beta\ge 0$, where
\[
  \varphi(\alpha,\beta) = e^{\beta^2/8}
    + e^{\alpha\beta/3+\beta^2/4+\alpha^4/4}
    - \alpha e^\beta + \tfrac23\alpha\beta
    + \alpha - 2.
\]
For $0\le\alpha\le 1$, we have
$\varphi(\alpha,\beta)\ge
  e^{\beta^2/8} + e^{\beta^2/4}
  - e^\beta - 2$,
which is positive when $\beta\ge 4$.
For $\alpha>1$, note that
$e^{\alpha^4/4}>\alpha$, so we have
$\varphi(\alpha,\beta)\ge
 (e^{\beta^2/4}-e^\beta)\alpha+e^{\beta^2/8}-2$,
and both coefficients are positive for $\beta>4$.
Thus, $\varphi(\alpha,\beta)\ge 0$ for $\alpha\ge 0,\beta\ge 4$.

For $0\le \beta\le 4$,
$\varphi(\alpha,\beta)\ge e^{\alpha^4/4}
 - \alpha e^4 + \alpha - 1$,
which is positive when $\alpha\ge 3$.

We are left with the rectangle
$R=\{(\alpha,\beta)\st 0\le\alpha\le 3, 0\le\beta\le 4\} $.
A polynomial $\bar\varphi(\alpha,\beta)$ such that
$\bar\varphi\le\varphi$ on~$R$ is obtained using the
bounds $e^x\ge 1+x+\tfrac12 x^2 + \tfrac16x^3$ for $x\ge 0$
and
$e^x\le 1+x+\tfrac12 x^2 + \tfrac16x^3 +
 \tfrac1{24}x^4 + \tfrac1{50}x^5$ for $0\le x\le 4$.
We will show that $\bar\varphi$ is nonnegative on~$R$.

Using Sturm sequences, we find that
 $\varphi(\alpha,\beta)>0$ everywhere on the boundary of~$R$
except at the point~$(0,0)$, where it is zero.
As $(\alpha,\beta)\to (0,0)$,
$\bar\varphi(\alpha,\beta) = (1+o(1))(\tfrac38\beta^2
  + \tfrac14\alpha^4)$,
which is positive in some punctured neighbourhood
of $(0,0)$.
Therefore, there is some $\eps>0$ such that $\bar\varphi(\alpha,\beta)\ge 0$ for
$0\le\alpha<\eps,0\le\beta\le 4$ and
$\bar\varphi(\alpha,\beta)> 0$ on the boundary of the rectangle
$R_\eps = \{ (\alpha,\beta)\st \eps\le\alpha\le 3, 0\le\beta\le 4\} $.
If $\bar\varphi(\alpha,\gamma)$ has a zero inside $R_\eps$, then there is some
constant $\alpha'\in(\eps,3)$ such that the 1-variable polynomial $\bar\varphi(\alpha',\gamma)$
has a multiple zero in $(0,4)$.
However the discriminant of $\bar\varphi(\alpha,\gamma)$
with respect to $\gamma$ is never zero for $0\le\alpha\le 4$.
\end{proof}

The following lemma is used in combining error terms.

\begin{lemma}\label{errorterms}
  Let $K_1,K_2,\eps_1,\eps_2\in\Complexes$ and
  $\alpha,\delta_1,\delta_2,\delta_3,\delta_4\ge 0$.
  Suppose $\abs{K_1}\le e^{\delta_1}-1$,
  $\abs{K_2}\le e^{\delta_2}-1$, $\abs{\eps_1}\le\delta_3$
  and $\abs{\eps_2}\le\delta_4$.  Then
  \[
     (1 + K_1e^{\alpha+\eps_2})(1+K_2)e^{\eps_1}
     = 1 + Ke^\alpha
  \]
  for some $K\in\Complexes$ with
  $\abs{K}\le e^{\delta_1+\delta_2+\delta_3+\delta_4}-1$.
\end{lemma}
\begin{proof}
  For $z\in\Complexes$ it is immediate from the Taylor series
  that $\abs{e^z}\le e^{\abs{z}}$ and $\abs{e^z-1}\le e^{\abs{z}}-1$.
  Bound $\abs{K}$ by bounding 
  $K=e^{-\alpha}\((1 + K_1e^{\alpha+\eps_2})
  (1+K_2)e^{\eps_1}-1\)$ term by term, which gives
  $\abs{K}\le e^{\delta_1+\delta_2+\delta_3+\delta_4}+
    e^{-\alpha+\delta_2+\delta_3} - e^{\delta_2+\delta_3+\delta_4}
    - e^{-\alpha}$.
  Therefore $e^{\delta_1+\delta_2+\delta_3+\delta_4}-1-
  \abs{K} \ge
  e^{\delta_2+\delta_3+\delta_4} + e^{-\alpha}
  - e^{-\alpha+\delta_2+\delta_3} -1
  \ge e^{\delta_2+\delta_3} + e^{-\alpha}
    - e^{-\alpha+\delta_2+\delta_3} -1$, which is nonnegative by
  the convexity of the exponential function since both 0 and
  $-\alpha+\delta_2+\delta_3$ lie in the interval $[-\alpha,
  \delta_2+\delta_3]$ and the average of 0 and
  $-\alpha+\delta_2+\delta_3$ lies at the midpoint of the interval.
\end{proof}

\begin{lemma}\label{series_exp}
	For any $m\in \mathbb{N}$ and $0\leq x\leq m/4$,
   \[
			\sum_{k=0}^{m-1} \frac{x^k}{k!} \leq e^x \leq \dfrac43\; \sum_{k=0}^{m-1} \frac{x^k}{k!}.
   \]
\end{lemma}
\begin{proof}
 The lower bound is clear. 
 For the upper bound note that by comparing terms
 \[
    e^x \le \sum_{k=0}^{m-1} \frac{x^k}{k!} + \frac{x^m}{m!}e^x,
 \]
 and so $e^x\le \(1-\frac{x^m}{m!}\)^{-1}\sum_{k=1}^{m-1}\frac{x^k}{k!}
 \le \(1-\frac{(m/4)^m}{m!}\)^{-1}\sum_{k=1}^{m-1}\frac{x^k}{k!}
 \le \frac43\,\sum_{k=1}^{m-1}\frac{x^k}{k!}$.
\end{proof}

The following lemma is an immediate corollary  of~\cite[Thm.~3]{switching}.
\begin{lemma}\label{switching}
  Let $D$ be a finite directed graph, with loops but not parallel edges
  allowed.
  Let $p:V(D)\to\Reals_{>0}$, $s:E(D)\to\Reals_{>0}$ and
  $\alpha:E(D)\to(0,1)$
  be functions such that the following inequalities hold.
  \begin{align*}
     \sum_{w:(vw)\in E(D)} \alpha(vw)s(vw) &\ge p(v),\qquad
         \text{for $v\in V(D)$ not a sink, and} \\
     \sum_{v:(vw)\in E(D)} s(vw) &\le p(w),\qquad
         \text{for all $w\in V(D)$}.
  \end{align*}
  Let $Z\subseteq V(D)$ be the set of sinks of~$G$.  Then
  \[
       \frac{\sum_{v\in V(G)-Z} p(v)}{\sum_{v\in V(G)} p(v)}
         \le \frac{\max_{(vw)\in E(D)} \alpha(vw)}{1-\max_{(vw)\in E(D)} \alpha(vw)}.
  \]
\end{lemma}

\nicebreak

\end{document}